\documentclass[11pt,a4paper,oneside]{article}

\usepackage{geometry}
\usepackage{hyperref} 

\usepackage{amssymb,amsmath}
\usepackage[amsthm,thmmarks]{ntheorem}
\usepackage{bbm}
\usepackage{enumitem}
\setenumerate{label={\normalfont(\alph*)},topsep=4pt,itemsep=0pt} 

\usepackage{float}
\usepackage{wrapfig}

\usepackage[T2A,T1]{fontenc}
\usepackage[utf8]{inputenc}
\usepackage[russian,UKenglish]{babel}
\newcommand{\SORTNOOPCYR}[1]{} 
\newcommand{\SortNoop}[1]{}

\usepackage{filecontents}
\usepackage{bibentry}
\nobibliography*

\usepackage{lmodern} 
\RequirePackage{fix-cm}

\usepackage{graphicx}
\usepackage{tikz}
\usetikzlibrary{matrix,arrows,decorations.pathmorphing}

\usepackage{wrapfig}
\usepackage{subfig}
\usetikzlibrary{calc}
\usepackage{relsize}
\usepackage{xfrac}

\usepackage{epstopdf}
\usepackage{comment}
\usepackage{color}

\usepackage{numname} 

\usepackage{mathtools}  

\theoremstyle{change}
\theoremseparator{}
\newtheorem{theorem}{Theorem}[section]
\newtheorem{lemma}[theorem]{Lemma}

\theoremstyle{change}
\theorembodyfont{\upshape}
\theoremseparator{}
\newtheorem{obs}[theorem]{\!\!}

\theoremstyle{change}
\theorembodyfont{\upshape}
\newtheorem{definition}[theorem]{Definition}

\newtheorem{remark}[theorem]{Remark}
\newtheorem{example}[theorem]{Example}
\newtheorem{examples}[theorem]{Examples} 

\theoremstyle{plain} 

\theoremstyle{margin}

\theoremstyle{definition}
\theorembodyfont{\color[rgb]{0,0.7,0}}

\theorembodyfont{\itshape\bfseries}
\makeatletter
\g@addto@macro\th@remark{\thm@headpunct{}}
\makeatother

\newcommand{\emphind}[1]{\emph{#1}\index{#1}}

\renewcommand{\epsilon}{\varepsilon}
\renewcommand{\P}{\mathbb{P}}
\newcommand{\1}{\mathbbm{1}}

\newcommand{\R}{\mathbb{R}}

\newcommand{\N}{\mathbb{N}}

\newcommand{\cA}{\mathcal{A}}
\newcommand{\cB}{\mathcal{B}}

\newcommand{\cE}{\mathcal{E}}
\newcommand{\cF}{\mathcal{F}}
\newcommand{\cG}{\mathcal{G}}
\newcommand{\cH}{\mathcal{H}}

\newcommand{\cM}{\mathcal{M}}
\newcommand{\cN}{\mathcal{N}}

\newcommand{\cP}{\mathcal{P}}

\newcommand{\cR}{\mathcal{R}}
\newcommand{\cS}{\mathcal{S}}

\newcommand{\cU}{\mathcal{U}}
\newcommand{\cV}{\mathcal{V}}

\newcommand{\cX}{\mathcal{X}}
\newcommand{\cY}{\mathcal{Y}}
\newcommand{\cZ}{\mathcal{Z}}

	\newcommand{\fd}{\mathfrak{d}}

\newcommand{\limn}{\lim_{n\rightarrow \infty}}

\newcommand{\liminfm}{\liminf_{m\rightarrow \infty}}
\newcommand{\liminfn}{\liminf_{n\rightarrow \infty}}

\newcommand{\limsupm}{\limsup_{m\rightarrow \infty}}
\newcommand{\limsupn}{\limsup_{n\rightarrow \infty}}

\newcommand{\supm}{\sup_{m\in \N}}

\DeclareMathOperator{\supp}{supp}

\DeclareMathOperator{\D}{\, \mathrm{d}\!}
\DeclareMathOperator{\DD}{\, \mathrm{d}}

\newcommand\putatop[2]{\genfrac{}{}{0pt}{}{#1}{#2}}

\newcommand{\assump}[1]{#1}

\newcommand{\rateindef}{\tfrac{1}{r_n}}
\newcommand{\traten}{\tfrac{1}{n}}
\newcommand{\tratem}{\tfrac{1}{m}}
\newcommand{\tratenm}{\tfrac{1}{n_m}} 

\DeclareMathOperator{\id}{Id}

\newcommand{\K}{K}
\newcommand{\J}{J}

\newcommand{\absnu}{|\nu|}
\newcommand{\absmu}{|\mu|}

\newcommand{\I}{\mathbb{I}}

\newcommand{\signed}{} 

\let\oldforall\forall
\renewcommand{\forall}{\ \oldforall}

\let\oldexists\exists
\renewcommand{\exists}{\ \oldexists}

\newcommand{\margi}{\mathfrak{m}}

\numberwithin{equation}{section}

\begin{document}


\title{Large deviations of 
continuous regular conditional probabilities}


\renewcommand{\thefootnote}{*}

\author{
W. van Zuijlen\footnotemark[1]
\\
}

\footnotetext[1]{
Mathematical Institute, Leiden University, P.O.\ Box 9512, 2300 RA, Leiden, The
Netherlands.
}

\renewcommand{\thefootnote}{\arabic{footnote}} 

\maketitle


\begin{abstract}
We study product regular conditional probabilities under measures of two coordinates with respect to the  second coordinate that are weakly continuous on the support of the marginal of the second coordinate. 
Assuming that there exists a sequence of probability measures on the product space that satisfies a large deviation principle, we present necessary and sufficient conditions for the conditional probabilities under these measures to satisfy a large deviation principle. 
The arguments of these conditional probabilities are assumed to converge. 
A way to view regular conditional probabilities as a special case of product regular conditional probabilities is presented. This is used to derive conditions for large deviations of regular conditional probabilities.
In addition, we derive a Sanov-type theorem for large deviations of the empirical distribution of the first coordinate conditioned on fixing the empirical distribution of the second coordinate.

\bigskip

\emph{Acknowledgement}. 
The author is supported by ERC Advanced Grant VARIS-267356 of Frank den Hollander.
The author is grateful to both Frank den Hollander and Frank Redig for valuable suggestions and useful discussions. 
\end{abstract}



\section{Introduction and main results}
\label{section:Intoduction_and_main_results}


In the present paper we study large deviations of probabilities ``of the form'' 
\begin{align}
\label{eqn:conditional_probability_of_RVS}
\P(X_n \in A \, | \, Y_n = y_n), 
\end{align}
where $((X_n,Y_n))_{n\in\N}$ is a sequence of couples of random variables that satisfies a large deviation principle and $y_n \rightarrow y$ for some $y$. 
As the event $[Y_n = y_n]$ may have probability zero, we make sense of \eqref{eqn:conditional_probability_of_RVS} in terms of a kernel $\eta_n$, so that 
$$\eta_n(y_n,A)$$
 ``represents'' \eqref{eqn:conditional_probability_of_RVS}. 

Such kernels are called regular conditional probabilities and form an important object in probability theory. 
The existence of regular conditional probabilities has been studied extensively, for example, by Faden \cite{Fa85} or by Leao, Fragoso and Fuffino \cite{LeFrRu04}. 
There exist in fact various forms of regular conditional probabilities; namely either 
with respect to a $\sigma$-algebra, 
with respect to a measurable map, 
or with respect to the projection on one of the coordinates (in case of a product space). 

In order to consider large deviations of conditional probabilities, we have to specify which conditional probability we are considering; the conditional probability may not be unique. 
However, if a (product) regular conditional probability is weakly continuous on the support of the measure composed with the inverse of the measurable map (or projection), it is unique on that domain. 
For these (product) regular conditional probabilities it is natural to study their large deviations, whenever the argument of the probability is in the domain on which it is unique. 
In this paper we study the large deviations in the case when the arguments of these kernels converge, i.e., we study large deviations of $(\eta_n(y_n,\cdot))_{n\in\N}$ for the case that $y_n \rightarrow y$. 
To the best of our knowledge, current literature 
does not provide a general condition under which such kernels satisfy a large deviation principle.

\subsection{Literature}

Some examples in this direction are present. 
For example in Adams, Dirr, Peletier and Zimmer \cite{AdDiPeZi11}, the large deviation principle is proved for the empirical distribution that is evolved by independent Brownian motions conditioned on their initial empirical distribution to lie in a ball (see \cite[Theorem 1]{AdDiPeZi11}).  
They proceed by proving that the large deviation principle rate function converges as the radius of the ball converges to zero. 
For the purpose of this paper, 
we have to show that the limit of the radius of the ball and the limit belonging to the large deviation principle can be interchanged. 
L\'{e}onard \cite{Le07}
proves the large deviation principle of the empirical distribution that is evolved by independent Brownian motions conditioned on their initial empirical distribution; 
those initial empirical distributions are assumed to be converging (see \cite[Proposition 2.19]{Le07}). 
In both papers, the evolved state is conditioned on the initial state, while there is also interest in large deviations of the initial state conditioned on the evolved state. 
In this paper we prove the large deviation principle in this setting for finite state spaces. 

There exist various results on quenched large deviations, i.e., large deviations for regular conditional probabilities in the sense that for almost all realisations of the disorder the conditional probabilities satisfy the large deviation principle with a rate function that does not depend on the disorder. 
Examples of papers on quenched large deviations are 
Comets \cite{Co89} for conditional large deviations of i.i.d. random fields,
Greven and den Hollander \cite{GrdH98} and Comets, Gantert and Zeitouni \cite{CoGaZe00} for random walks in random environments,
Kosygina-Rezakhanlou-Varadhan \cite{KoReVa06}, for a diffusion with a random drift, 
Rassoul-Agha, Sepp{\"a}l{\"a}inen and Yilmaz \cite{RaSeYi13} for polymers in a random potential. 



Biggins \cite{Bi04} obtains the large deviation principle for mixtures of probability measures that satisfy the large deviation principle with kernels that satisfy the large deviation principle as their arguments converge.
To some extent we complement the article in the opposite direction, in the sense that we assume the large deviation principle of the mixture and derive the large deviation principle of the kernels.

Our main motivation to study the above large deviations lies in the theory of Gibbs-non-Gibbs transitions. 
There is a correspondence between the large deviation rate function of the conditional probability with respect to the evolved coordinate and the evolved state (measure or sequence) being Gibbs (see van Enter, Fern\'{a}ndez, den Hollander and Redig \cite{vEFedHoRe10}). 
We refer to Section \ref{subsection:GNG_discussion} for further discussions on Gibbs-non-Gibbs transitions.

\subsection{Large deviations}
\label{subsection:LDP}

In the literature on large deviations two dominant definitions of large deviation principles are used. One is  in terms of a $\sigma$-algebra on the topological space, as is done in the book by Dembo and Zeitouni \cite{DeZe10} and in the book by Deuschel and Stroock \cite{DeSt84}, the other is  in terms of the topology, i.e., in terms of open and closed sets, as is done in the book by den Hollander \cite{dH00} and in the book by Rassoul-Agha and Sepp\"al\"ainen \cite{RaSe15}.
Whenever one considers the Borel-$\sigma$-algebra on the topological space, the two definitions agree. 

We define the  large deviation lower bound and the  large deviation upper bound separately, as in 
Section \ref{subsection:main_results} and in 
Section \ref{section:ldp_product_RCP} we describe the necessary and sufficient conditions for each of the bounds separately. Moreover, we define them on a set of subsets of the topological space, which is not required to be a $\sigma$-algebra. 
In Remark \ref{remark:motivation_LDP_definition} we motivate the choice for this definition. 

\begin{definition}
\label{def:ldp_adapted}
Let $\cX$ be a topological space and $\cA$ be a set of subsets of $\cX$.  
Let $I: \cX \rightarrow [0,\infty]$ be lower semicontinuous. 
Let $(\mu_n)_{n\in\N}$ be a sequence of probability measures on $\cA$. 
Let $(r_n)_{n\in\N}$ be an increasing sequence in $(0,\infty)$ with $\limn r_n = \infty$. 
We say that $(\mu_n)_{n\in\N}$ satisfies a 
\emphind{large deviation lower bound}
on $\cA$ 
with rate function $I$ and rates $(r_n)_{n\in\N}$  if 
\begin{align}
\label{eqn:def_ldp_lower_bound}
\liminfn \rateindef \log \mu_n(A) \ge - \inf I(A^\circ) \qquad (A\in \cA). 
\end{align}
We say that $(\mu_n)_{n\in\N}$ satisfies a 
\emphind{large deviation upper bound}
on $\cA$ 
with rate function $I$ and rates $(r_n)_{n\in\N}$ if 
\begin{align}
\label{eqn:def_ldp_upper_bound}
\limsupn \rateindef \log \mu_n(A) \le - \inf I(\overline A) \qquad (A\in \cA). 
\end{align}
\assump{In the rest of the paper we only consider the rates $r_n=n$.} 
However, the theory presented is still valid for general rates $(r_n)_{n\in\N}$.
We say that $(\mu_n)_{n\in\N}$ satisfies a 
\emphind{large deviation principle}
on $\cA$ 
 with rate function $I$ 
 whenever it satisfies both the large deviation lower bound and the large deviation upper bound with rate function $I$. 
 
We omit ``on $\cA$'' whenever $\cA$ is the Borel-$\sigma$-algebra $\cB(\cX)$ on $\cX$.
In this case the large deviation lower bound is satisfied if and only if the inequality in \eqref{eqn:def_ldp_lower_bound} holds for all open subsets of $\cX$ and the large deviation upper bound is satisfied if and only if the inequality in \eqref{eqn:def_ldp_upper_bound} holds for all closed subsets of $\cX$. 
\end{definition}

\subsection{Main results}
\label{subsection:main_results}

See Section \ref{section:rcp} and Section \ref{section:weak_continuous_cps} for the definitions of the objects in the statements of the following theorems. 
In Section \ref{section:ldp_product_RCP} and Section \ref{section:ldp_RCP} we consider a more general situation. Theorem \ref{theorem:equivalent_notions_ldp_bounds_metric_PRCP} 
is a consequence of Theorem \ref{theorem:equivalent_notions_bounds_prodRCP} and 
Theorem \ref{theorem:equivalent_notions_ldp_bounds_metric_RCP} is a consequence of Theorem \ref{theorem:equivalent_notions_bounds_RCP}. 

\assump{In this section $\cX$ and $\cY$ are a metric spaces.} 

\begin{theorem}
\label{theorem:equivalent_notions_ldp_bounds_metric_PRCP}
Let $\pi : \cX \times \cY \rightarrow \cY$ be given by $\pi(x,y)=y$. 
Suppose that
$(\mu_n)_{n\in\N}$ is a sequence of probability measures on $\cB(\cX) \otimes \cB(\cY)$ that satisfies the large deviation principle with rate function $J: \cX \times \cY \rightarrow [0,\infty]$ that has compact sublevel sets. 
Suppose that for each $n\in\N$ there exists a product regular conditional probability 
 $\eta_n : \cY \times \cB(\cX) \rightarrow [0,1]$ under $\mu_n$ with respect to $\pi$ 
that is weakly continuous on $\supp( \mu_n \circ \pi^{-1})$, which we assume to be nonempty.
Let $y \in \cY$ be such that $\inf J(\cX \times \{y\})<\infty$. 
Define $I: \cX \rightarrow [0,\infty]$ by 
\begin{align}
I(x) = J(x,y) - \inf J(\cX \times \{y\}). 
\end{align}
$I$ has compact sublevel sets, and, for each $n\in\N$, $\eta_n$ is unique on $\supp(\mu_n \circ \pi^{-1})$. 
Moreover,
\begin{center}
\ref{item:forall_y_seq_cond_lower_bound_metric} $\iff$ \ref{item:uniform_condition_lower_bound_metric} and 
\ref{item:forall_y_seq_cond_upper_bound_metric} $\iff$ \ref{item:uniform_condition_upper_bound_metric},
\end{center}
where
\begin{enumerate}[label=\emph{(A\arabic*)}]
\item 
\label{item:forall_y_seq_cond_lower_bound_metric}
For all $(y_n)_{n\in\N}$ with $y_n \rightarrow y$ and $y_n \in \supp(\mu_n\circ \pi^{-1})$  for all $n$ large enough,\footnote{Meaning that there exists an $N\in\N$ such that $y_n \in \supp(\mu_n\circ \pi^{-1})$ for all $n\ge N$.}
the sequence $(\eta_n(y_n,\cdot))_{n\in\N}$ satisfies the large deviation lower bound with rate function $I$. 
\item 
\label{item:uniform_condition_lower_bound_metric}
For all $x\in \cX$ and $r>0$,
with $U= B(x,r)$,
\begin{align}
\label{eqn:metric_equiv_cond_all_seq_lower_ldp}
\sup_{\epsilon>0}  \liminfn 
\inf_{ \putatop{z\in \cY, \delta \in (0,\epsilon) }{ B(z,\delta) \subset B(y,\epsilon) } }
\traten \log \mu_n \Big( \overline{U} \times \cY  \, \Big| \, \cX \times B(z,\delta) \Big) \ge -\inf I( U ). 
\end{align}
\end{enumerate}
\begin{enumerate}[label=\emph{(B\arabic*)}]
\item 
\label{item:forall_y_seq_cond_upper_bound_metric}
For all $(y_n)_{n\in\N}$ with $y_n \rightarrow y$ and $y_n \in \supp(\mu_n\circ \pi^{-1})$  for all $n$ large enough,
the sequence $(\eta_n(y_n,\cdot))_{n\in\N}$ satisfies the large deviation upper bound with rate function $I$. 
\item 
\label{item:uniform_condition_upper_bound_metric}
For all $x_1,\dots,x_k\in \cX$ and $r_1,\dots,r_k>0$,  with $W = \cX \setminus [B(x_1,r_1)\cup \cdots \cup B(x_k,r_k)]$,
\begin{align}
\label{eqn:metric_equiv_cond_all_seq_upper_ldp}
\inf_{\epsilon>0}  \limsupn  
\sup_{ \putatop{z\in \cY, \delta \in (0,\epsilon) }{ B(z,\delta) \subset B(y,\epsilon) } }
\traten \log \mu_n \Big( W^\circ \times \cY \, \Big| \, \cX \times B(z,\delta) \Big) \le -\inf I(W ). 
\end{align}
\end{enumerate}
\end{theorem}

The next theorem is similar to Theorem \ref{theorem:equivalent_notions_ldp_bounds_metric_PRCP}, 
but considers the large deviation bounds for regular conditional kernels instead of product regular conditional probabilities. 

\begin{theorem}
\label{theorem:equivalent_notions_ldp_bounds_metric_RCP}
Let $\tau: \cX \rightarrow \cY$ be continuous. 
Suppose that
$(\nu_n)_{n\in\N}$ is a sequence of probability measures on $\cB(\cX) $ that satisfies the large deviation principle with rate function $J: \cX  \rightarrow [0,\infty]$ that has compact sublevel sets. 
Suppose that for each $n\in\N$ there exists a regular conditional probability
 $\eta_n : \cY \times \cB(\cX) \rightarrow [0,1]$ under $\nu_n$ with respect to $\tau$ that is weakly continuous on $\supp( \nu_n \circ \tau^{-1})$, which is assumed to be nonempty. 
Let $y \in \cY$ be such that $\inf J( \tau^{-1}(\{y\}))<\infty$. 
Define $I: \cX \rightarrow [0,\infty]$ by 
\begin{align}
I(x) =
\begin{cases}
 J(x) - \inf J(\tau^{-1}(\{y\})) & \tau(x)=y, \\
 \infty & \tau(x) \ne y. 
\end{cases}
\end{align}
$I$ has compact sublevel sets, and, for each $n\in\N$, $\eta_n$ is unique on $\supp(\nu_n \circ \tau^{-1})$. 
Moreover,
\begin{center}
\ref{item:forall_y_seq_cond_lower_bound_RCP_metric} $\iff$ \ref{item:uniform_condition_lower_bound_RCP_metric} and 
\ref{item:forall_y_seq_cond_upper_bound_RCP_metric} $\iff$ \ref{item:uniform_condition_upper_bound_RCP_metric}, 
\end{center}
where
\begin{enumerate}[label=\emph{(A\arabic*)}]
\item 
\label{item:forall_y_seq_cond_lower_bound_RCP_metric}
For all $(y_n)_{n\in\N}$ with $y_n \rightarrow y$ and $y_n \in \supp(\nu_n\circ \pi^{-1})$  for all $n$ large enough,
the sequence $(\eta_n(y_n,\cdot))_{n\in\N}$ satisfies the large deviation lower bound with rate function $I$. 
\item 
\label{item:uniform_condition_lower_bound_RCP_metric}
For all $x\in \cX$ and $r>0$, 
with $U= B(x,r)$,
\begin{align}
\label{eqn:metric_equiv_cond_all_seq_lower_ldp_for_RCP}
\sup_{\epsilon>0}  \liminfn 
\inf_{ \putatop{z\in \cY, \delta \in (0,\epsilon) }{ B(z,\delta) \subset B(y,\epsilon) } }
\traten \log \nu_n \Big( \overline{U}  \ \Big| \, \tau^{-1}( B(z,\delta) ) \Big) \ge -\inf I( U ). 
\end{align}
\end{enumerate}
\begin{enumerate}[label=\emph{(B\arabic*)}]
\item 
\label{item:forall_y_seq_cond_upper_bound_RCP_metric}
For all $(y_n)_{n\in\N}$ with $y_n \rightarrow y$ and $y_n \in \supp(\nu_n\circ \pi^{-1})$  for all $n$ large enough,
the sequence $(\eta_n(y_n,\cdot))_{n\in\N}$ satisfies the large deviation upper bound with rate function $I$. 
\item 
\label{item:uniform_condition_upper_bound_RCP_metric}
For all $x_1,\dots,x_k\in \cX$ and $r_1,\dots,r_k>0$,  with $W = \cX \setminus [B(x_1,r_1)\cup \cdots \cup B(x_k,r_k)]$,
\begin{align}
\label{eqn:metric_equiv_cond_all_seq_upper_ldp_for_RCP}
\inf_{\epsilon>0}  \limsupn  
\sup_{ \putatop{z\in \cY, \delta \in (0,\epsilon) }{ B(z,\delta) \subset B(y,\epsilon) } }
\traten \log \nu_n \Big( W^\circ  \, \Big| \, \tau^{-1}(B(z,\delta) ) \Big) \le -\inf I(W ). 
\end{align}
\end{enumerate}
\end{theorem}

\subsection{Gibbs-non-Gibbs transitions and future research}
\label{subsection:GNG_discussion}

In this section we discuss the relation between the large deviation results in this paper and Gibbs-non-Gibbs transitions in more detail. In particular, we discuss possible future directions regarding large deviations of conditional kernels. 

The following situation for interacting particle systems occurs in the mean-field context (a similar context holds in the context of lattices). 
The initial system of -so called- spins consists of distributions describing the interaction between spins via a potential $V$ (for each $n$ there is a distribution describing the law of $n$ spins). 
This initial system is assumed to be Gibbs, which is called sequentially Gibbs in the mean-field context. 
Allowing the initial state to be transformed, for example, by an evolution of the spins, a question of interest is whether the transformed state is (sequentially) Gibbs. 
This question has been addressed in the mean-field context by Ermoleav and K\"{u}lske \cite{ErKu10} and by Fern\'{a}ndez, den Hollander and Mart\'{i}nez \cite{FedHoMa13} for $\{-1,+1\}$-valued spins, by den Hollander, Redig and van Zuijlen \cite{dHRevZ} for $\R$-valued spins and by K\"{u}lske and Opoku \cite{KuOp08} and van Enter, K{\"u}lske, Opoku and Ruszel \cite{vEKuOpRu10} for compactly valued spins. 
In these papers, independent dynamics of the spins are considered (the evolution of each spin is independent of the evolution of the other spins). 
Independent dynamics simplify the situation. 
Namely, the evolved measure on either the product space of the initial and the final space, or -in case of an evolution- the space of trajectories, is a tilted measure of the evolved measure when considering $V=0$. In this case the measure is a product measure, which means that the spins are independent. 
As a consequence (this will be clarified in a forthcoming paper) the conditional kernel $\eta_n$ 
of the initial state on $n$ spins with respect to the final state (for a fixed potential $V$) is a tilted version of the conditional kernel $\eta_n^0$ of the initial state with respect to the final state of independent spins (i.e., $V=0$). 
Because of this tilting, by Varadhan's lemma, $(\eta_n(y_n,\cdot))_{n\in\N}$ satisfies the large deviation principle with rate function $V+ I_y - \inf (V+ I_y)$ if $(\eta_n^0(y_n,\cdot))_{n\in\N}$ satisfies the large deviation principle with rate function $I_y$. 
In the forthcoming paper we will prove that the evolved sequence is sequentially Gibbs if $V+I_\zeta$ has a unique global minimiser. 

The large deviation principle of $(\eta_n(y_n,\cdot))_{n\in\N}$ has been mentioned in the case of trajectories in 
\cite[Corollary 2.4]{ErKu10} 
and -as a corollary of that theorem- for the case of the product space of the initial and the final space in
\cite[Corollary 1.3]{FedHoMa13}. However, no proof was given. 
Theorem \ref{theorem:application_sanov_type} provides a rigorous proof of the large deviation principle statement in \cite[Corollary 1.3]{FedHoMa13}. 
In this paper we do not provide a rigorous proof of \cite[Corollary 2.4]{ErKu10}. 
But Theorem \ref{theorem:equivalent_notions_ldp_bounds_metric_RCP} may be used, as the conditioning on the final state is a regular conditional kernel with respect to the map $\tau: C([0,T],\cX) \rightarrow \cX$, $\tau(f)=f(T)$. 

In order to deal with empirical distributions (and not with magnetisations as is done in \cite{dHRevZ}), in future research we strive to `extend' the statement of Theorem \ref{theorem:application_sanov_type} to infinite and possibly non-compact state spaces. 
In the case of non-compact spaces it may be that topologies on the space of probability measures are considered that are not metrisable.

\subsection{Outline}

We list some notations, definitions and assumptions in Section \ref{section:notation_and_conventions_cond_LDP}. 
In Section \ref{section:rcp} we give and compare the notions of regular conditional kernels, we show that a regular conditional kernel under a measure $\nu$ is in fact a product regular conditional kernel under a measure that is related to $\nu$. 
In Section \ref{section:weak_continuous_cps} we introduce and study weakly continuous regular conditional kernels. 
In  Section \ref{section:facts_about_lsc_functions_with_compact_level_sets} we present some facts about lower semicontinuous functions with compact sublevel sets. 
Relying on the results of Sections \ref{section:weak_continuous_cps} and \ref{section:facts_about_lsc_functions_with_compact_level_sets}, in Section \ref{section:ldp_product_RCP} we present results on large deviation bounds for product regular conditional probabilities, in particular, necessary and sufficient conditions for these bounds to hold. 
In Section \ref{section:ldp_RCP} we discuss how to obtain large deviation bounds for regular conditional probabilities from the results in Section \ref{section:ldp_product_RCP}. 
In Section \ref{section:application_sanov_type} we apply the theory to obtain the large deviation principle for the empirical density of the first coordinate given the empirical density of the second coordinate, for independent and identically distributed pairs of random variables. 
In Section \ref{section:examples} 
we give some examples. 
We also include an example for which the conditions are not satisfied. 
For this example we compare the quenched large deviations with large deviations of the weakly continuous regular conditional probabilities and comment on the difference with an example by La Cour and Schieve \cite{CoSc15}. 
In appendices \ref{appendix:elemtary_fact} and \ref{appendix:sufficient_bounds} we state some general results considering large deviations bounds that are used in the different sections. 
In appendix \ref{section:proof_theorem_examples} we provide the proof of a theorem on which the examples of Section \ref{section:examples} rely.

\section{Notations and conventions}
\label{section:notation_and_conventions_cond_LDP}

$\N=\{1,2,3,\dots\}$. 
For a topological space $\cX$ we write $\cB(\cX)$ for the Borel-$\sigma$-algebra and $\cP(\cX)$ and $\cM(\cX)$ for the spaces of probability and signed measures on $\cB(\cX)$, respectively.  
For $A\subset \cX$ we write $A^\circ$ for the interior of $A$ and $\overline A$ for the closure of $A$. 
For $x\in \cX$ we write $\delta_x$ for the element in $\cP(\cX)$ with $\delta_x(A) =1$ if $ x\in A$ and $\delta_x(A)=0$ otherwise.
For $x\in \cX$ we write $\cN_x$ for the set of $\cB(\cX)$-measurable neighbourhoods of $x$. 
For a $\mu\in \cM(\cX)$ we write $\supp \mu = \{ x\in \cX: |\mu|(V)>0 \mbox{ for all } V\in \cN_x\}$ and call this the support of $\mu$. 
For a function $f$ from a set $\cX$ into $\R$ and $c\in \R$ we write $[f \ge c] = \{x\in \cX: f(x) \ge c\}$. Similarly, we use the notations $[f > c]$, $[f\le c]$ and $[f<c]$. 
Whenever $(x_\iota)_{\iota\in \I}$ is a net, where $\I$ is a directed set by (a direction) $\preceq$, we write $\liminf_{\iota \in \I} x_\iota = \sup_{\iota_0\in \I} \inf_{\iota \succeq \iota_0, \iota \in \I} x_\iota$ (similarly $\limsup$). 
In particular, if $\cV\subset \cN_x$ and $\bigcap \cV = \{x\}$ and $f: \cV \rightarrow \R$ we write $\liminf_{V\in \cV} f(V) = \sup_{V_0\in \cV} \inf_{V\subset V_0, V\in \cV} f(V)$ (i.e., we consider $(f(V))_{V\in \cV}$ as a net where $\cV$ is directed by $\supset$ (as $\preceq$)).

Whenever we write $\mu(A| B)$ we implicitly assume that it is well defined (as $\mu(A\cap B)/\mu(B)$), i.e., that $\mu(B)\ne 0$. 

We use the conventions $\log 0 = - \infty$ and $\inf I(\emptyset) = \infty$ whenever $I$ is a function with values in $[0,\infty]$. 

\assump{All measures in this paper are signed measures, unless mentioned otherwise.
}

\section{Regular conditional kernels being product regular conditional kernels}
\label{section:rcp}

In this section we introduce the notion of a (product) regular conditional kernel. For an extensive study on regular conditional kernels see Bogachev \cite[Section  10.4]{Bo07}. The notion of a product regular conditional kernel does not appear in \cite{Bo07}, but it does in Faden \cite{Fa85} and in Leao, Fragoso and Fuffino \cite{LeFrRu04}. 
Besides giving definitions we make a few observations, of which Theorem \ref{theorem:extension_and_RCP_as_special_case_of_PRCP} 
is used later on to derive statements of regular conditional kernels from statements of product regular conditional kernels. 

\assump{
In this section $(X,\cA)$, $(Y,\cB)$ are measurable spaces,
$\nu$ is a \signed  measure on $\cA$ and $\mu$ is a \signed  measure on $\cA \otimes \cB$, 
$\tau: X \rightarrow Y$ is measurable and  $\pi : X \times Y \rightarrow Y$ is  given by $\pi(x,y)=y$. 
 }

\begin{definition}
A function $\eta: Y \times \cA \rightarrow \R$ is called a 
($\cB$-)\emphind{kernel} 
 if  
 $\eta(\cdot ,A)$ is \mbox{($\cB$-)}measur\-able 
  for all $A\in \cA$ and $\eta(y,\cdot)$ is a \signed  measure for all $y \in Y$.
A kernel $\eta$ is called a \emphind{probability kernel} if $\eta(y,\cdot)$ is a probability measure for all $y \in Y$. 
\end{definition}

\begin{definition}
\label{def:rcps}
Let $\eta: Y \times \cA \rightarrow \R$ be a (probability) kernel. 
\begin{enumerate}
\item 
\label{item:def_rcp_function}
$\eta$ is called a 
\emphind{regular conditional kernel} (\emphind{regular conditional probability}) under $\nu$ with respect to $\tau$
if
\begin{align}
\label{eqn:def_rcp_function}
\nu(F \cap \tau^{-1}(B) ) = \int_Y \1_B(y) \eta(y,F) \DD \left[\absnu \circ \tau^{-1}\right](y) \qquad (F\in \cF, B\in \cB).
\end{align}
\item 
\label{item:def_rcp_product}
$\eta$ is called a 
\emphind{product regular conditional kernel} (\emphind{product regular conditional probability})
under $\mu$ with respect to $\pi$ if
\begin{align}
\label{eqn:def_rcp_product}
\mu(A \times B ) = \int_Y \1_B(y) \eta(y,A) \DD \left[\absmu \circ \pi^{-1}\right](y) \qquad (A\in \cA, B\in \cB).
\end{align}
\end{enumerate}
\end{definition}

\begin{obs}
\label{obs:RCP_wrt_sigma_algebra}
Suppose that $\cE$ is a sub-$\sigma$-algebra of $\cF$. 
Let $(Y,\cB) = (X, \cE)$ and $\id : (X, \cA) \rightarrow (Y,\cB)$ be the identity map. 
In agreement of \cite[Definition 10.4.1]{Bo07} 
a kernel $\eta: Y \times \cA \rightarrow \R$ is a 
\emphind{regular conditional kernel} under $\mu$ with respect to $\cE$ if and only if $\eta$ is a regular conditional kernel under $\mu$ with respect to $\id$. 
\end{obs}

\begin{obs}
\label{obs:rcps_compared}
Consider the two kernels $\eta: Y \times \cA \rightarrow \R$ and $\xi: Y \times ( \cA \otimes \cB) \rightarrow \R$, corresponding to each other by the formulas $\xi (y, F) = \int_\cX \1_{F}(x,y) \DD [\eta(y,\cdot)](x)$ and $\eta(y,A) = \xi(y,A\times Y)$. 
Then $\xi$ is a regular conditional kernel under $\mu$ given $\pi$ if and only if $\eta$ is a product regular conditional kernel under $\mu$ given $\pi$. 

In general, $X\times Y$ may be equipped with a $\sigma$-algebra $\cF$ different from $\cA \otimes \cB$. In this situation, where $\mu$ is a \signed measure 
on $\cF$ and $\pi$ is $\cF$-measurable the above  correspondence cannot be used in general to reduce statements about product regular conditional kernels to  statements about regular conditional kernels. See also example \ref{obs:comparison_product_and_function_rcp_for_unicity}. 

On the other hand, regular conditional probabilities can be seen as special cases of product regular conditional probabilities, see Theorem \ref{theorem:extension_and_RCP_as_special_case_of_PRCP}. In the present paper we use this to derive Theorem \ref{theorem:equivalent_notions_ldp_bounds_metric_RCP} from Theorem \ref{theorem:equivalent_notions_ldp_bounds_metric_PRCP} but also Theorem \ref{theorem:equivalent_notions_bounds_RCP} from Theorem \ref{theorem:equivalent_notions_bounds_prodRCP}. 
\end{obs}

\begin{remark}
\label{remark:almost_everywhere_uniqueness}
If $\cA$ is generated by a countable set, two
regular conditional probabilities 
under a \signed measure 
with respect to a $\sigma$-algebra (see \ref{obs:RCP_wrt_sigma_algebra}) 
are almost everywhere equal (see Bogachev \cite[Theorem 10.4.3]{Bo07}).
Similarly one could state an analogues statement for regular conditional kernels with respect to measurable maps and for product regular conditional kernels. 
In Theorem \ref{theorem:unicity_weakly_continuous_rcp_wrt_pi} we prove that (product) regular conditional kernels are unique on the domain on which they are weakly continuous, in case the underlying topological space is perfectly normal. 
For such space the Borel-$\sigma$-algebra may not be generated by a countable set.\footnote{The Sorgenfrey line, the space $\R$ with the right half-open interval topology, is perfectly normal but not second countable (see Steen and Seebach \cite[Example 51]{StSe70}).}
\end{remark}

\begin{theorem}
\label{theorem:extension_and_RCP_as_special_case_of_PRCP}
\begin{enumerate}
\item 
\label{item:extension_measure_to_product}
There exists a \signed  measure $\tilde \mu$ on $(X \times Y, \cA \otimes \cB)$ for which 
$\tilde \mu(A\times B) = \nu( A \cap \tau^{-1}(B))$. 
\item 
\label{item:RCP_as_PRCP}
$\eta: Y\times \cA \rightarrow \R$ is a regular conditional kernel under $\nu$ with respect to $\tau$ if and only if $\eta$ is a product conditional kernel under $\tilde \mu$ with respect to $\pi$. 
\end{enumerate}
\end{theorem}
\begin{proof}
\ref{item:extension_measure_to_product}
We may assume $\nu$ to be positive, since $\nu = \nu^+ - \nu^-$. 
Let $\cE$ be the set that consists of $\bigcup_{i=1}^n A_i \times B_i$, where $n\in\N$ and $A_i \in \cA, B_i \in \cB$ are such that $A_1\times B_1,\dots, A_n\times B_n$ are disjoint. 
Define $\nu^*: \cE \rightarrow [0,\infty)$ by 
$
\nu^* \left( \bigcup_{i=1}^n A_i \times B_i \right)  = \nu \left( \bigcup_{i=1}^n A_i \cap \tau^{-1}(B_i) \right )
$ for $A_1,\dots,A_n\in \cA$ and $B_1,\dots,B_n $ $ \in \cB$ as above.  
Checking that $\cE$ is a ring of sets and that $\nu^*$ is $\sigma$-additive is left for the reader. 
The existence and unicity of the extension $\tilde \mu$ follows from the Carath\'eodory Theorem (see Halmos \cite[Section  13, Theorem A]{Ha74}). \\
\ref{item:RCP_as_PRCP}
Follows from by definition of $\tilde \mu$ (note that $\nu \circ \tau^{-1} = \tilde \mu \circ \pi^{-1}$). 
\end{proof}

\section{Weakly continuous kernels}
\label{section:weak_continuous_cps}

In this section we introduce the notion of weak continuity for kernels on topological spaces. 
In Theorem \ref{theorem:unicity_weakly_continuous_rcp_wrt_pi} we show uniqueness of (product) regular conditional kernels that are weakly continuous. 
In Theorem \ref{theorem:discrete_top_PRCP} and Theorem  \ref{theorem:mixing_with_function_gives_weakly_continuous_kernel} 
we describe conditions that imply the existence of weakly continuous regular conditional probabilities. 
Similarly as is done in the Portmanteau Theorem when one considers metric spaces, weak convergence implies lower bounds for open sets and upper bounds for closed sets, as is shown in Theorem \ref{theorem:weak_convergence_equivalences}. 
As described in Lemma \ref{lemma:liminf_and_limsup_of_conditional_open_compared_with_kernel_PRCP}
and Lemma \ref{lemma:liminf_and_limsup_of_conditional_open_compared_with_kernel_RCP}
these $\liminf$ and $\limsup$ bounds imply bounds
for (product) regular conditional probabilities on which the results of 
Sections \ref{section:ldp_product_RCP} and \ref{section:ldp_RCP} are based.

\assump{In this section $\cX$ and $\cY$ are topological spaces,
$\nu$ is a \signed  measure on $\cB(\cX)$ and $\mu$ is a \signed  measure on $\cB(\cX)\otimes \cB(\cY)$,
$\tau : \cX \rightarrow \cY$ is measurable and $\pi : \cX \times \cY \rightarrow \cY$ is given by $\pi(x,y) = y$. 
}

\begin{definition}
\label{def:weakly_continuous_transition_prob}
We equip the space of \signed  measures, $\cM(\cX)$,  with the weak topology (generated by $C_b(\cX)$, and denoted by $\sigma(\cM(\cX),C_b(\cX))$ as in the book of Schaefer \cite[Chapter II, Section 5]{Sc70}). In this topology, a net $(\mu_\iota)_{\iota\in \I}$ in $\cM(\cX)$ converges to a $\mu$ in $\cM(\cX)$ if 
$\int_\cX f \D \mu_\iota \rightarrow \int_\cX f \D \mu $ for all $f\in C_b(\cX)$. \\
Let $D\subset \cY$. 
A kernel $\eta : \cY \times \cB(\cX) \rightarrow \R$ is called 
\emphind{weakly continuous}
on $D$ if the map $D\rightarrow \cM(\cX)$ given by $y\mapsto \eta(y,\cdot)$ is continuous in the weak topology. 
$\eta$ is called 
\emph{weakly continuous}
 if $\eta$ is weakly continuous on $\cY$.
\end{definition}

\begin{theorem}
\label{theorem:measures_on_perfectly_normal_space_have_support}
Let $\cX$ be a perfectly normal\footnote{Perfectly normal means that every open set in $\cX$ is equal to $f^{-1}((0,\infty))$ for some $f\in C(\cX)$. All metric spaces are perfectly normal; Bogachev \cite[Proposition 6.3.5]{Bo07}.}  space and $\mu \in \cM(\cX)$. 
Then 
\begin{align}
\supp \mu = \Big\{ x\in \cX :  \int_\cX f \D \absmu >0 \mbox{ for all }  f\in C(\cX,[0,1]) \mbox{ with } f(x)>0 \  \Big\}.
\end{align}
Moreover, $|\mu|( \cX \setminus \supp(\mu)) =0$.\footnote{This is not true in general. For an example see Bogachev \cite[Example 7.1.3]{Bo07}.} 
As a consequence, $\mu =0$ if and only if $\int_\cX f \D \absmu =0$ for all $f\in C_b(\cX)$. 
\end{theorem}
\begin{proof} 
We may assume $\mu$ is positive. 
Let $x\in \supp \mu$. 
Then $\mu(V)>0$ for all $V\in \cN_x$. 
Let $f\in C(\cX,[0,1])$ be such that $f(x)>0$. 
Then $V= f^{-1}(0,\infty)$ has strictly positive measure.
Since $\mu(V) = \limn \int_\cX \min \{nf, 1\} \D \mu$, there exists an $n$ such that $\int_\cX \min \{nf, 1\} \D \mu>0$. Consequently, as $f \ge  \frac1n \min \{nf, 1\}$, we have $\int_\cX f \D \mu>0$. 

Let $x\in \cX$ be such that $\int_\cX f \D \mu >0$ for all $f \in C(\cX,[0,1])$ with $f(x)>0$. 
Let $V\in \cN_x$. 
As $V= f^{-1}(0,\infty)$ for some $f\in C(\cX,[0,1])$, we have $\mu(V) \ge \int_\cX f \D \mu >0$. 
\end{proof}

\begin{theorem}
\label{theorem:unicity_weakly_continuous_rcp_wrt_pi}
Suppose that $\cX$ is a perfectly normal space.
\begin{enumerate}
\item Let $\eta$ and $\zeta$ be regular conditional kernels under $\nu$ with respect to $\tau$ that are weakly continuous on $\supp (\absnu \circ \tau^{-1})$.
Then $\eta(y,\cdot)=\zeta(y,\cdot)$ for all $y\in \supp (\absnu \circ \tau^{-1})$. 
If $\nu$ is a probability measure, then $\eta(y,\cdot)$ is a probability measure for all $y\in \supp(\absnu \circ \tau^{-1})$.
\item Let $\eta$ and $\zeta$ be product regular conditional kernels under $\mu$ with respect to $\pi$ that are weakly continuous on $\supp(\absmu \circ \pi^{-1})$. Then $\eta(y,\cdot)=\zeta(y,\cdot)$ for all $y\in \supp (\absmu \circ \pi^{-1})$. 
If $\mu$ is a probability measure, then $\eta(y,\cdot)$ is a probability measure for all $y\in \supp(\absmu \circ \pi^{-1})$.
\end{enumerate}
\end{theorem}
\begin{proof}
We prove (a), the proof of (b) is similar (replace ``$\absnu \circ \tau^{-1}$'' by ``$\absmu \circ \pi^{-1}$''). 
To prove $\eta = \zeta$ on $D=\supp (\absnu \circ \tau^{-1})$, by Theorem \ref{theorem:measures_on_perfectly_normal_space_have_support}, it is sufficient to prove $\int_\cX f \D\eta(y,\cdot) = \int_\cX f \D \zeta(y,\cdot)$ for all $y\in D$ and all $f\in C_b(\cX)$. 
Let $f\in C_b(\cX)$. 
Because $f$ is the uniform limit of simple functions, one has for all $B\in \cB(\cY)$
\begin{align}
\int_\cY \1_B(y) \left[ \int_\cX f \D \eta(y,\cdot) \right] \DD [ \absnu \circ \tau^{-1} ](y)
= \int_\cY \1_B(y) \left[ \int_\cX f \D \zeta(y,\cdot) \right] \DD [ \absnu \circ \tau^{-1} ](y).
\end{align}
Therefore there exists a set 
$Z\in \cB(\cY)$ with $\absnu \circ \tau^{-1}(\cY \setminus Z)=0$  such that 
\begin{align}
\int_\cX f \D \eta(z,\cdot) = \int_\cX f \D \zeta(z,\cdot) \qquad (z\in Z). 
\end{align}
Since both $y \mapsto \int_\cX f \D\eta(y,\cdot)$ and $y \mapsto \int_\cX f \D\zeta(y,\cdot)$ are weakly continuous on $D$, and $Z$ is dense in $D$ by 
Theorem \ref{theorem:measures_on_perfectly_normal_space_have_support},
 we have $\int_\cX f \D\eta(y,\cdot) = \int_\cX f \D \zeta(y,\cdot)$ for all $y\in D$. 
The second statement is proved by taking $f=\1_\cX$. 
\end{proof}

\begin{obs}
When $\eta$ is a regular conditional kernel under $\nu$ with respect to $\tau$, the value of the function $\eta(\cdot,A)$ on the complement of $\supp(\absnu \circ \tau^{-1})$ is not determined, in the sense that, if $\tilde \eta$ is a kernel with $\tilde \eta(y,\cdot) = \eta(y,\cdot)$ for all $y\in \supp(\absnu \circ \tau^{-1})$, then $\tilde \eta$ is also a regular conditional kernel under $\nu$ with respect to $\tau$. 

For example $\tilde \eta$ given by $\tilde \eta(y,\cdot) = \eta(y,\cdot)$ for $y\in \supp(\absnu \circ \tau^{-1})$ and $\tilde \eta(y,\cdot) = \delta_x$ for $y\in \supp(\absnu \circ \tau^{-1})^c$ for some chosen $x\in \cX$, is such regular conditional kernel. 

Whence if $\nu$ is a probability measure and there exists a regular conditional kernel under $\nu$ with respect to $\tau$ that is weakly continuous on $\supp(\absnu \circ \tau^{-1})$, then we may as well assume this kernel to be a probability kernel. A similar statement is true for product regular conditional kernels. 
\end{obs}

\begin{obs}
\label{obs:comparison_product_and_function_rcp_for_unicity}
By Theorem \ref{theorem:extension_and_RCP_as_special_case_of_PRCP} statement (a) of Theorem \ref{theorem:unicity_weakly_continuous_rcp_wrt_pi} is a consequence of statement (b). 
In an attempt to reduce statement (b) to statement (a) the following problem occurs to the correspondence between regular conditional kernels and product regular conditional kernels that is mentioned in \ref{obs:rcps_compared}. 

The Borel-$\sigma$-algebra of $\cX \times \cY$, i.e., $\cB(\cX \times \cY)$ may be strictly larger as $\cB(\cX) \otimes \cB(\cY)$ (see, e.g., Bogachev \cite[Lemma 6.4.1 and Example 6.4.3]{Bo07}). 
If this is the case, i.e., $\cB(\cX)\otimes \cB(\cY) \subsetneq \cB(\cX \times \cY)$, and $\cB(\cX \times \cY)$ equals the Baire-$\sigma$-algebra on $\cX \times \cY$, i.e., the smallest $\sigma$-algebra that makes all continuous function $\cX \times \cY \rightarrow \R$ measurable; then there exists a continuous function $f\in C(\cX \times \cY)$ that is not $\cB(\cX) \otimes \cB(\cY)$-measurable. 
Composing the function $f$ with $\arctan$, we obtain a $g\in C_b(\cX \times \cY)$ that is not measurable with respect to $\cB(\cX)\otimes \cB(\cY)$. 
So if $\eta: \cY \times \cB(\cX) \rightarrow \R$ is a product regular conditional kernel under $\mu$ with respect to $\pi$, and $\xi : \cY \times \cB(\cX) \otimes \cB(\cY) \rightarrow \R$ is as in Example \ref{obs:rcps_compared} then $g$ is not integrable with respect to $\xi(y,\cdot)$ for any $y\in \cY$. 

$\cB(\cX \times \cY)$ equals the Baire-$\sigma$-algebra if $\cX\times \cY$ is a metric space (Bogachev \cite[Proposition 6.3.4]{Bo07}). 
Therefore $\cX = \cY = \R^\R$ equipped with the discrete topology form an example for which the above is the case. 
\end{obs}

We state two theorems (Theorem \ref{theorem:discrete_top_PRCP} and Theorem \ref{theorem:mixing_with_function_gives_weakly_continuous_kernel}) showing the existence of product regular conditional probabilities that are weakly continuous on $\supp( \absmu \circ \pi^{-1})$. 

\begin{theorem}
\label{theorem:discrete_top_PRCP}
Suppose that $\cY$ is countable and equipped with the discrete topology. 
Then $\eta: \cY \times \cB(\cX) \rightarrow \R$ defined by 
\begin{align}
\eta(y,A) = 
\begin{cases}
\mu(A\times \cY| \cX \times \{y\}) & \mu(\cX \times \{y\}) \ne 0, \\
0	& \mu(\cX \times \{y\})=0, 
\end{cases}
\end{align}
is a product regular conditional kernel under $\mu$ with respect to $\pi$ that is weakly continuous on $\supp (\absmu \circ \pi^{-1})$. 
\end{theorem}
\begin{proof}
Follows from the fact that $\mu(A\times B) = \sum_{y\in B}  \mu(A\times \{y\})$ for $A\in \cB(\cX)$, $B\in \cB(\cY)$. 
\end{proof}

The following theorem is an easy consequence of Lebesgue's Dominated
Convergence Theorem. 

\begin{theorem}
\label{theorem:mixing_with_function_gives_weakly_continuous_kernel}
Let $\lambda$ be a probability measure on $\cB(\cX)$. 
Let $D\subset \cY$. 
Let $f: \cX \times \cY \rightarrow [0,\infty)$ be a bounded $\cB(\cX) \otimes \cB(\cY)$-measurable function such that $y\mapsto f(x,y)$ is continuous on $D$ and equal to zero on $\cY \setminus D$ for $\lambda$-almost all $x\in \cX$. 
Suppose that $\int_\cX f(x,y) \D \lambda(x)>0$ for all $y\in D$. 
If $\eta\colon\,
\cY \times \cB(\cX) \rightarrow [0,1]$ is given by 
\begin{align}
 \eta(y,A) =
\begin{cases} 
 \frac{\int_\cX \1_A(x) f(x,y) \D \lambda(x)}{\int_\cX  f(x,y) \D \lambda(x)} & y\in D, \\
0 & y \notin D. 
\end{cases}
\end{align}
then $\eta$ is weakly continuous on $D$ (even strongly continuous, i.e., $y \mapsto \eta(y,A)$ is continuous for all $A\in \cB(\cX)$). 
Let $\kappa$ be a probability measure on $\cB(\cY)$ and assume $D= \supp \kappa$. 
Then $\eta$ is a product regular conditional kernel under 
\begin{align}
\mu: \cB(\cX) \otimes \cB(\cY) \rightarrow [0,1], \qquad 
\mu(A)=
\frac{\int_{\cX \times \cY} \1_A  f  \DD [\lambda \otimes \kappa]}{\int_{\cX \times \cY}  f  \DD [\lambda \otimes \kappa]}
\end{align}
with respect to $\pi$, that is weakly continuous on $D= \supp(\absmu\circ \pi^{-1})$. 
\end{theorem}

\begin{obs}
\label{obs:first_countable_remarks}
In case $\cY$ is first countable, the notion of open and closed sets and continuity of functions $\cY \rightarrow \R$ is characterised by the convergence of sequences. Therefore the following are equivalent for a kernel $\eta : \cY \times \cB(\cX) \rightarrow \R$ 
\begin{enumerate}
\item $\eta$ is weakly continuous in $y$.
\item For all $(y_n)_{n\in\N}$ in $\cY$ with $y_n \rightarrow y$ one has $\eta (y_n,\cdot) \xrightarrow{w} \eta(y,\cdot)$.
\end{enumerate}
In Section \ref{section:ldp_product_RCP} the condition \ref{item:eta_weakly_continuous_open_sets_net} of Theorem \ref{theorem:weak_convergence_equivalences} is one of the key assumptions. 
If $\cX$ is a metric space, this property follows from the weak continuity as in the Portmanteau Theorem. We state this in Theorem \ref{theorem:weak_convergence_equivalences}. 
\end{obs}

\begin{theorem}
\label{theorem:weak_convergence_equivalences}
Let $\eta : \cY \times \cB(\cX) \rightarrow \R$ be a probability kernel. 
Let $D\subset \cY$, $y\in D$ and $\cV\subset \cN_y$ be such that $\bigcap \cV = \{y\}$.
Consider the following conditions. 
\begin{enumerate}
\item \label{item:eta_weakly_continuous}
$D \rightarrow \cM(\cX)$, $y\mapsto \eta(y,\cdot)$ is weakly continuous in $y$. 
\item \label{item:eta_weakly_continuous_open_sets_net}
$\liminf_{\iota\in\I} \eta(y_\iota,G) \ge \eta(y,G)$ for all open $G\subset \cX$ and $(y_\iota)_{\iota\in\I}$ in $D$ with
$y_\iota \rightarrow y$.
\item \label{item:eta_weakly_continuous_closed_sets_net}
$\limsup_{\iota\in\I}\eta(y_\iota,F) \le \eta(y,F)$ for all closed $F\subset \cX$ and $(y_\iota)_{\iota\in\I}$ in $D$ with
$y_\iota \rightarrow y$.

\item \label{item:eta_weakly_continuous_open_sets_neighbourhoods}
$\sup_{V\in \cV} \inf_{v\in V\cap D} \eta(v,G) \ge \eta(y,G)$ for all open sets $G\subset \cX$.

\item \label{item:eta_weakly_continuous_closed_sets_neighbourhoods}
$\inf_{V\in \cV} \sup_{v\in V\cap D} \eta(v,F) \le \eta(y,F)$ for all closed sets $F\subset \cX$.

\end{enumerate}
\ref{item:eta_weakly_continuous_open_sets_net},
\ref{item:eta_weakly_continuous_closed_sets_net},
\ref{item:eta_weakly_continuous_open_sets_neighbourhoods},
\ref{item:eta_weakly_continuous_closed_sets_neighbourhoods} are equivalent. 
If $\cX$ is metrisable, then \ref{item:eta_weakly_continuous} implies \ref{item:eta_weakly_continuous_open_sets_net}. 
If $\cX$ is metrisable and $\cY$ is first countable, then
\ref{item:eta_weakly_continuous}
is equivalent to \ref{item:eta_weakly_continuous_open_sets_net} and hence to 
\ref{item:eta_weakly_continuous_closed_sets_net},
\ref{item:eta_weakly_continuous_open_sets_neighbourhoods} and
\ref{item:eta_weakly_continuous_closed_sets_neighbourhoods}.
\end{theorem}
\begin{proof}
We leave it to the reader to check the equivalences between \ref{item:eta_weakly_continuous_open_sets_net},
\ref{item:eta_weakly_continuous_closed_sets_net},
\ref{item:eta_weakly_continuous_open_sets_neighbourhoods},
\ref{item:eta_weakly_continuous_closed_sets_neighbourhoods}. 
If $\cX$ is a metric space, one can follow the lines of the Portmanteau Theorem in the book of Billingsley \cite[Theorem 2.1]{Bi68} for the implication \ref{item:eta_weakly_continuous} implies \ref{item:eta_weakly_continuous_open_sets_net}, the fact that the measures in the proof are indexed by the natural numbers instead of a general directed set $\I$ does not affect the argument. 
The proof of \ref{item:eta_weakly_continuous_open_sets_net}$\Longrightarrow$\ref{item:eta_weakly_continuous} in the book of Billingsley 
 relies on the Lebesgue Dominated Convergence theorem. But when $\cY$ is first countable, one can restrict to sequences (see \ref{obs:first_countable_remarks}) and obtain the implication \ref{item:eta_weakly_continuous_open_sets_net}$\Longrightarrow$\ref{item:eta_weakly_continuous} as is done in
the book of Billingsley. 
\end{proof}

\begin{lemma}
\label{lemma:liminf_and_limsup_of_conditional_open_compared_with_kernel_PRCP}
Assume that $\mu$ is a probability measure. 
Let $\eta$ be a product regular conditional probability under $\mu$ with respect to $\pi$.
Write $D=\supp(\mu \circ \pi^{-1})$ and let $y\in D$. 
Then for every $U\in \cN_y$ one has $\mu(\cX \times U) >0$ and 
\begin{align}
\label{eqn:conditional_on_open_between_inf_and_sup_kernel}
\inf_{v\in U\cap D} \eta(v,A) \le \mu(A \times \cY | \cX \times U) \le   \sup_{v\in U\cap D } \eta (v,A)  
 \qquad (A \in \cB(\cX)). 
\end{align}
Moreover, if $\cV\subset \cN_y$ is such that $\bigcap \cV = \{y\}$ and $\eta$ satisfies \ref{item:eta_weakly_continuous_open_sets_net} of Theorem \ref{theorem:weak_convergence_equivalences}, then
\begin{align}
\label{eqn:conditional_on_open_giving_upper_bound}
\liminf_{V\in \cV}
\mu(G \times \cY | \cX \times V) &\ge \eta(y,G) \mbox{ for all open } G\subset \cX,  \\
\label{eqn:conditional_on_closed_giving_lower_bound}
\limsup_{V\in \cV}
\mu(F \times \cY | \cX \times V) &\le \eta(y,F) \mbox{ for all closed } F\subset \cX.
\end{align}
\end{lemma}
\begin{proof}
Let $U \in \cN_y$. 
Since $y\in D=\supp(\mu \circ \pi^{-1})$ one has $ \mu(\cX \times U)>0$. 
\eqref{eqn:conditional_on_open_between_inf_and_sup_kernel} follows from the fact that for all $A\in \cB(\cX)$
\begin{align}
\frac{ \mu (A \times U) }{ \mu (\cX \times U)}  
\notag & = 
\frac{ \int_\cY \1_U(y) \eta(y,A) \DD [\mu \circ \pi^{-1}](y) }{  \int_\cY \1_U(y) \DD [\mu \circ \pi^{-1}](y) } \\
& =
\frac{ \int_\cY \1_{U\cap D}(y) \eta(y,A) \DD [\mu \circ \pi^{-1}](y) }{  \int_\cY \1_{U\cap D}(y) \DD [\mu \circ \pi^{-1}](y) } .
\end{align}
For an open  $G\subset \cX$ we have for $\cV$ as above
\begin{align}
\liminf_{V\in\cV}
\mu(G \times \cY | \cX \times V) 
\ge \liminf_{V\in\cV} \inf_{v\in V \cap D} \eta(v,G)
= \sup_{V\in\cV} \inf_{v\in V\cap D} \eta(v,G) . 		
\end{align}
Thus \eqref{eqn:conditional_on_open_giving_upper_bound} follows when assuming \ref{item:eta_weakly_continuous_open_sets_net} of Theorem \ref{theorem:weak_convergence_equivalences}. Similarly, one obtains \eqref{eqn:conditional_on_closed_giving_lower_bound}.
\end{proof}

For a regular conditional probability we have a similar statement, see Lemma \ref{lemma:liminf_and_limsup_of_conditional_open_compared_with_kernel_RCP}. 
The proof can be done following the lines of the proof of Lemma \ref{lemma:liminf_and_limsup_of_conditional_open_compared_with_kernel_PRCP} or as a consequence of Lemma \ref{lemma:liminf_and_limsup_of_conditional_open_compared_with_kernel_PRCP} using Theorem \ref{theorem:extension_and_RCP_as_special_case_of_PRCP}. 

\begin{lemma}
\label{lemma:liminf_and_limsup_of_conditional_open_compared_with_kernel_RCP}
Assume that $\nu$ is a probability measure. 
Let $\eta$ be a regular conditional probability under $\nu$ with respect to $\tau$.
Write $D=\supp(\nu \circ \tau^{-1})$ and let $y\in D$. 
Then for every $U\in \cN_y$ one has $\nu( \tau^{-1}(U)) >0$ and 
\begin{align}
\label{eqn:conditional_on_open_between_inf_and_sup_kernel_function}
\inf_{v\in U\cap D} \eta(v,A) \le \nu( A | \tau^{-1}( U) ) \le 
  \sup_{v\in U\cap D } \eta (v,A)  \qquad (A \in \cB(\cX)). 
\end{align}
Moreover, if $\cV\subset \cN_y$ is such that $\bigcap \cV = \{y\}$ and $\eta$ satisfies \ref{item:eta_weakly_continuous_open_sets_net} of Theorem \ref{theorem:weak_convergence_equivalences}, then
\begin{align}
\liminf_{V\in\cV}
\nu(G  | \tau^{-1}(V) ) &\ge \eta(y,G)  \mbox{ for all open } G\subset \cX,  \\
\limsup_{V\in\cV}
\nu(F  | \tau^{-1}( V) ) &\le \eta(y,F)  \mbox{ for all closed } F\subset \cX.
\end{align}
\end{lemma}

\section{Some facts about functions with compact sublevel sets}
\label{section:facts_about_lsc_functions_with_compact_level_sets}

In this section we present some facts for functions with compact sublevel sets which are used in Sections \ref{section:ldp_product_RCP}, \ref{section:ldp_RCP} and \ref{section:application_sanov_type}.

\assump{In this section $\cX,\cY$ and $\cZ$ are topological spaces.}

\begin{definition}
Let $J: \cX \rightarrow [0,\infty]$. 
We call the set $[J\le \alpha]$ (see Section \ref{section:notation_and_conventions_cond_LDP}) a sublevel set of $J$ for $\alpha \in [0,\infty)$. 
$J$ is said to be \emphind{lower semicontinuous} if all sublevels of $J$ are closed. 
$J$ is said to have \emphind{compact sublevel sets} if all sublevels of $J$ are compact. 
\end{definition}

\begin{obs}
\label{obs:lower_semicontinuity_elementary_properties}
Let $J: \cX \rightarrow [0,\infty]$ be lower semicontinuous. 
Then 
\begin{align}
J(x) = \sup_{G\in \cN_x} \inf J(G).
\end{align}
Indeed, for all $\alpha < J(x)$ the set $[J > \alpha ] $ is open and contains $x$. 

Hence, 
a function $J: \cX \rightarrow [0,\infty]$ is lower semicontinuous if and only if 
\begin{align}
\liminf_{\iota \in \I} J(x_\iota) \ge J(x)
\end{align}
for all $x\in \cX$ and all nets $(x_\iota)_{\iota \in \I}$ in $\cX $ that converge to $x$.
\end{obs}

\begin{lemma}
\label{lemma:convergence_infimum_rate_function_on_schrinking_neighbourhoods}
Let $\tau : \cZ \rightarrow \cY$ be continuous. 
Let $J: \cZ \rightarrow [0,\infty]$ have compact sublevel sets. 
Let $y\in \cY$ and $\cV\subset \cN_y$, $\bigcap \cV = \{y\}$. 
Let $F\subset \cZ$ be closed. 
Then 
\begin{align}
\label{eqn:limit_of_infima_of_rate_function_of_smaller_neighbourhoods}
\liminf_{V\in \cV} \inf J(F \cap \tau^{-1}(\overline V) ) = \inf J(F \cap \tau^{-1}(\{y\}) ) .
\end{align}
Consequently, if $\cZ = \cX \times \cY$, then, for all closed $F\subset \cX$ with $\inf J(F \times \{y\})<\infty$,
\begin{align}
\liminf_{V\in \cV} \inf J(F\times  \overline V) = \inf J(F \times \{y\} ).
\end{align}
\end{lemma}
\begin{proof}
The $\le$ inequality in  \eqref{eqn:limit_of_infima_of_rate_function_of_smaller_neighbourhoods} is immediate. 
Because $\liminf_{V\in \cV} \inf J(F \cap \tau^{-1}( \overline V ) ) \ge \liminf_{V\in \cN_y} \inf J(F \cap \tau^{-1}( \overline{V} ) ) $, it is sufficient to prove  
\begin{align}
\alpha:= \liminf_{V\in \cN_y} \inf J(F \cap \tau^{-1}( \overline{V}) ) \ge \inf J(F \cap \tau^{-1}( \{y\}) ).
\end{align}
Note that $\alpha =\sup_{V\in \cN_y} \inf J(F \cap \tau^{-1}(\overline{V}) )$. 
If $\alpha = \infty$, there is nothing to prove. 
Suppose that $\alpha <\infty$. 
Whence $F\cap \tau^{-1}(\overline{V}) \cap [ J \le \alpha + \epsilon] \ne \emptyset$ for all $V\in \cN_y$ and all $\epsilon>0$. 
Since $[J \le \alpha + \epsilon]$ is compact, this implies that $\bigcap_{V\in \cN_y} F\cap \tau^{-1}(\overline{V}) \cap [ J \le \alpha + \epsilon] \ne \emptyset$, i.e., $\inf J(F \cap \tau^{-1}(\{y\})) \le \alpha + \epsilon$ for all $\epsilon>0$. 
\end{proof}

\begin{obs}
\label{obs:non_redundant_conditions_lemma_convergence_inf_rf_on_schrinking_nbhb}
The assumption that $\tau$ be continuous is not redundant; e.g., consider $\cY=\cZ =[0,1]$ and $J = \1_{(\frac12,1]}$ and $\tau$ given by $\tau(0)=0$, $\tau(1)=1$ and $\tau(x) = 1-x$ for $x\in (0,1)$, $F=[0,1]$ and $y=1$. Then, for all neighbourhoods $V$ of $y$, $\tau^{-1}(V)$ contains the interval $(0,\epsilon)$ for some $\epsilon>0$, whence $\inf J(F\cap \tau^{-1}(V)) = 0$ but $\inf J(F\cap \tau^{-1}(\{y\})) = J(1) = 1$. 
\end{obs}

\begin{lemma}
\label{lemma:normal_space_small_difference_with_smaller_open_set}
Let $\cX$ be normal and let $\cG$ be a basis for the topology of $\cX$. 
Let $J: \cX \times \cY \rightarrow [0,\infty]$ have compact sublevel sets. 
\begin{enumerate}
\item For all open $G\subset \cX$ and $\epsilon>0$ 
there exists a $U\in \cG$ with $U\subset \overline U \subset G$ such that 
\begin{align}
\inf J(G \times \{y\}) + \epsilon 
\ge \inf J (U \times \{y\}). 
\end{align}
\item For all closed $F\subset \cX$ and $\alpha < \inf J(F \times \{y\})$,  
 there exists $U_1,\dots, U_k \in \cG$ such that with $W=\cX \setminus ( U_1\cup \cdots \cup U_k)$ one has  $F\subset W^\circ \subset W$ and
\begin{align}
\alpha
 &  < 
\inf J( W \times \{y\})  
 \le \inf J( W^\circ  \times \{y\}) \le \inf J(F \times \{y\}).
\end{align}

\end{enumerate}
\end{lemma}
\begin{proof}
(a) Let $\epsilon>0$. 
Let $x\in G$ be such that 
$
J(x,y) \le \inf J(G \times \{y\}) + \epsilon. 
$ 
Since $\cX$ is a normal topological space, there exists an open set $U$ with $x\in U \subset \overline U \subset G$. Because $\cG$ is a basis, $U$ may be chosen in $\cG$. 
Then 
$
\inf J(G \times \{y\}) + \epsilon \ge J(x,y) \ge \inf J(U \times \{y\}).
$

(b) 
Let $\beta>\alpha$ be such that $\beta <  \inf J(F \times \{y\})$. 
The set $K:=\{ x\in \cX: J(x,y)\le \beta\}$ is a compact set that is disjoint from $F$. 
Whence there exists disjoint open $U,V \subset \cX$ with $K\subset U$ and $F\subset V$. 
Since $\cG$ is a basis and $K$ is compact, there exists $U_1,\dots,U_k$ in $\cG$ with $K\subset U_1\cup \cdots \cup U_k \subset U$. 
Then $\overline{ U_1\cup \cdots \cup U_k} \cap V = \emptyset$.
 Whence  with $W:=\cX \setminus \overline{ U_1\cup \cdots \cup U_k}$ one has $F\subset W^\circ$ and $W\subset \cX \setminus K$, which implies $\inf J(W \times \{y\}) \ge \beta > \alpha$. 
\end{proof}

\section{Large deviations for product regular conditional probabilities}
\label{section:ldp_product_RCP}

\assump{
In this section we consider the following situation. 
\begin{enumerate}[label=(\roman*)]
\setitemize{leftmargin=0pt}
\setlength{\itemsep}{0pt}
  \setlength{\parskip}{0pt}
  \setlength{\parsep}{0pt}
\item $\cX$ and $\cY$ are topological spaces, where $\cX$ is normal.
\item 
\label{item:basis_topologies}
$\cG$ is a basis for the topology of $\cX$ and $\cH$ is a basis for the topology of $\cY$. 
\item $\pi : \cX \times \cY \rightarrow \cY$ is given by $\pi(x,y) =y$. 
\label{item:pi_map}
\item 
\label{item:ass_ldp_mu_n}
$(\mu_n)_{n\in\N}$ is a sequence of probability measures on 
$\cB(\cX) \otimes \cB(\cY)$
satisfying the large deviation principle on $\{A\times B: A\in \cB(\cX), B\in \cB(\cY)\}$ with a rate function $J: \cX \times \cY \rightarrow [0,\infty]$ that has compact sublevel sets. 
\item 
\label{item:existence_prcp_mun_wrt_pi}
For each $n\in\N$ we assume the following: 
$\supp ( \mu_n \circ \pi^{-1}) \ne \emptyset$,\footnote{As we are considering large deviation bound for $(\eta_n(y_n,\cdot))_{n\in\N}$ with $y_n \in \supp( \mu_n \circ \pi^{-1})$ we want such $y_n$ to exist. Instead of this condition one could of course deal with the situation where $\supp ( \mu_n \circ \pi^{-1}) \ne \emptyset$ for all $n\ge N$ for some large $N$ and consider sequences $(y_n)_{n\in\N}$ with $y_n  \in \supp( \mu_n \circ \pi^{-1})$ for $n\ge N$.}
there exists a product regular conditional probability  $\eta_n : \cY \times \cB(\cX) \rightarrow [0,1]$ under $\mu_n$ with respect to $\pi$, which satisfies the following continuity condition (see Theorem \ref{theorem:weak_convergence_equivalences}):
\begin{align}
\notag & \mbox{$\liminf_{\iota\in\I} \eta_n(y_\iota,G) \ge \eta_n(y,G)$ for all open $G\subset \cX$} \\
\label{eqn:continuity_assumption}
& \mbox{and $(y_\iota)_{\iota\in\I}$ in $\supp(\mu_n \circ \pi^{-1}) $ with
$y_\iota \rightarrow y$.} 
\end{align}
\item 
\label{item:y_finiteness_rate_function_and_def_I}
Let $y \in \cY$. We assume that $\inf J(\cX \times \{y\})<\infty$ and that there exist $y_n \in \supp( \mu_n \circ \pi^{-1})$ with $y_n \rightarrow y$. 
We define $I: \cX \rightarrow [0,\infty]$ by 
\begin{align}
I(x) = J(x,y) - \inf J(\cX \times \{y\}). 
\end{align}
\end{enumerate}
}

In this section we derive necessary and sufficient conditions for the large deviation bounds with rate function $I$ for sequences of the form $(\eta_n(y_n,\cdot))_{n\in\N}$. 
We prove this for general topological spaces instead of metric spaces as it does not cost more effort. 

In Theorem \ref{theorem:lower_and_upper_bound_eta_n_convergent_y_n} we consider a fixed  sequence $(y_n)_{n\in\N}$ with $y_n \rightarrow y$ and describe equivalent conditions for the lower and upper large deviation bound to hold. 

We are interested in the question whether for all sequences $(y_n)_{n\in\N}$ with $y_n \rightarrow y$ the sequence $(\eta_n(y_n,\cdot))_{n\in\N}$ satisfies the lower and upper large deviation bound with rate function $I$.  
In Theorem \ref{theorem:equivalent_notions_bounds_prodRCP}
we give equivalent\footnote{Under the condition that $\cY$ is first countable.}
and sufficient conditions for these bounds 
in a way that does not depend on sequences $(y_n)_{n\in\N}$ and the sets $(\cV_n)_{n\in\N}$ as in Theorem \ref{theorem:lower_and_upper_bound_eta_n_convergent_y_n}.

Finally in \ref{obs:deriving_the_main_theorems} we comment on deriving Theorem \ref{theorem:equivalent_notions_ldp_bounds_metric_PRCP} from Theorem \ref{theorem:equivalent_notions_bounds_prodRCP}. 

But first we consider specific situations, providing a simple proof of the large deviation bounds with rate function $I$ for sequences of the form  $(\eta_n(y_n,\cdot))_{n\in\N}$. 
Namely, we consider the case that $\cY$ is a discrete space (Theorem \ref{theorem:countable_discrete_Y}) and the case where $\mu_n$ is a product measure for all $n\in\N$ (Theorem \ref{theorem:independent_coordinates}). 

\begin{theorem}
\label{theorem:countable_discrete_Y}
Suppose that $\cY$ is countable and equipped with the discrete topology.
Let $y\in \cY$ be such that $\inf J(\cX \times \{y\})<\infty$. 
For all $(y_n)_{n\in\N}$ in $\cY$ with $y_n \in \supp( \mu_n \circ \pi^{-1})$ and $y_n \rightarrow y$ the sequence $(\eta_n(y_n,\cdot))_{n\in\N}$ satisfies the large deviation principle with rate function $I$. 
\end{theorem}
\begin{proof}
This basically follows from the following inequalities which follow from the large deviation principle 
and from Theorem \ref{theorem:discrete_top_PRCP}. 
\begin{align}
\liminfn \traten \log \mu_n( G \times \{y\}) \ge - \inf J(G\times \{y\})  \mbox{ for all open } G \subset \cX, \\
\limsupn \traten \log \mu_n(F \times \{y\}) \le - \inf J(F \times \{y\}) \mbox{ for all closed } F \subset \cX. 
\end{align}
\
\end{proof}

\begin{theorem}[Independent coordinates]
\label{theorem:independent_coordinates}
Suppose that $\cX$ and $\cY$ are second countable and $\cY$ is regular. 
Suppose that  $\mu_n = \mu_n^1 \otimes \mu_n^2$
for some $\mu_n^1$ on $\cB(\cX)$ and $\mu_n^2$ on $\cB(\cY)$  for all $n\in\N$. 
Then $(\eta_n(y_n,\cdot))_{n\in\N}$ 
satisfies the large deviation principle with rate function $I$ for all sequences $(y_n)_{n\in\N}$ in $\cY$. 
In particular, 
$\eta_n(y_n,\cdot) = \mu_n^1$ and $I(x) = \inf J( \{x\} \times \cY)$. 
\end{theorem}
\begin{proof}
It is straightforward to see that $\eta_n(y,\cdot) = \mu_n^1$ for all $y\in \cY$. 
$(\mu_n^1)_{n\in\N}$ satisfies the large deviation principle with rate function $J_1(x) := \inf J(\{x\} \times \cY)$. Indeed, for an open set $G\subset \cX$ and a closed set $F\subset \cX$ we have 
\begin{align}
\liminfn \traten \log \mu_n^1(G) = \liminfn \traten \log ( \mu_n^1 \otimes \mu_n^2)(G\times \cY)
\ge -\inf J(G\times \cY), \\
\limsupn \traten \log \mu_n^1(F) = \limsupn \traten \log ( \mu_n^1 \otimes \mu_n^2)(F\times \cY)
\le -\inf J(F\times \cY).
\end{align}
Similarly, $(\mu_n^2)_{n\in\N}$ satisfies the large deviation principle with rate function $J_2(z) := \inf J(\cX \times \{z\})$. 
$J_1$ and $J_2$ are lower semicontinuous, which can be concluded by  \ref{obs:lower_semicontinuity_elementary_properties} and Lemma \ref{lemma:convergence_infimum_rate_function_on_schrinking_neighbourhoods}, as for example, $\lim_{\iota \in \I} z_\iota = z$ implies $\liminf_{\iota \in \I} \inf J(\cX \times \{z_\iota\}) \ge \liminf_{V\in \cN_z} \inf J(\cX \times V)$. 
Using Theorem \ref{theorem:reducing_ldp_bounds_to_smaller_set} it is not difficult to prove that $(\mu_n)_{n\in\N}$ satisfies the large deviation principle with rate function $(x,z) \mapsto J_1(x) + J_2(z)$, so that (see Rassoul-Agha and Sepp{\"a}l{\"a}inen \cite[Theorem 2.18]{RaSe15}) $J(x,z) = J_1(x) + J_2(z)$, and thus $I(x) = J(x,y) - \inf J(\cX \times \{y\})= J_1(x) = \inf J(\{x\} \times \cY)$ for all $x\in \cX$, $z\in \cY$. 
\end{proof}

\begin{theorem}
\label{theorem:lower_and_upper_bound_eta_n_convergent_y_n}
Let $(y_n)_{n\in\N}$ be a sequence in $\cY$ with $y_n \in \supp(\mu_n \circ \pi^{-1})$ that converges to $y$. 
For $n\in\N$ let $\cV_n \subset \cN_{y_n}$ be such that $\bigcap \cV_n = \{y_n\}$. 
Then 
\ref{item:condition_lower_bound_basis} $\iff$ \ref{item:condition_lower_bound_stronger} $\iff$ \ref{item:ldp_lower_bound} 
and 
\ref{item:condition_upper_bound_basis} $\iff$ \ref{item:condition_upper_bound_stronger} $\iff$\ref{item:ldp_upper_bound}
\begin{enumerate}[label={\normalfont(a\arabic*)},topsep=4pt,itemsep=0pt]
\item
\label{item:ldp_lower_bound}
For all open $G\subset \cX$ 
\begin{align}
\label{eqn:lower_bound_eta_n_fixed_y}
\liminfn \traten \log \eta_n(y_n,G) \ge - 
\inf I(G). 
\end{align}
\item 
\label{item:condition_lower_bound_basis}
For all $U\in \cG$\footnote{Note that $\mu_n(\cX \times V) >0$ for all $n\in\N$ and $V\in \cN_{y_n}$, as $y_n \in \supp(\mu_n \circ \pi^{-1})$.}
\begin{align} 
\liminfn \limsup_{V\in\cV_n} \traten \log \mu_n( \overline U \times \cY | \cX \times V) 
\label{eqn:lower_condition} \ge - 
\inf I(U). 
\end{align}
\item 
\label{item:condition_lower_bound_stronger}
For all open $U\subset \cX$ 
one has
\begin{align} 
\label{eqn:condition_lower_bound_stronger}
\liminfn \liminf_{V\in\cV_n} \traten \log \mu_n( U \times \cY | \cX \times V) 
\ge - 
\inf I(U). 
\end{align}
\end{enumerate}
\begin{enumerate}[label={\normalfont(b\arabic*)},topsep=4pt,itemsep=0pt]
\item
\label{item:ldp_upper_bound}
For all closed $F\subset \cX$ 
\begin{align}
\label{eqn:upper_bound_eta_n_fixed_y}
\limsupn \traten \log \eta_n(y_n,F) \le - 
\inf I(F).
\end{align}
\item
\label{item:condition_upper_bound_basis}
For all $U_1,\dots,U_k\in \cG$ one has for $W= \cX \setminus (U_1\cup \cdots \cup U_k)$
\begin{align} 
\limsupn \liminf_{V\in\cV_n} \traten \log \mu_n( W^\circ \times \cY | \cX \times V) 
\le  - 
\inf I(W).
\label{eqn:upper_condition} 
\end{align}
\item 
\label{item:condition_upper_bound_stronger}
For all closed $W\subset \cX$ 
\begin{align} 
\label{eqn:condition_upper_bound_stronger}
\limsupn \limsup_{V\in\cV_n} \traten \log \mu_n( W \times \cY | \cX \times V) 
\le  - 
\inf I(W).
\end{align}
\end{enumerate}
\end{theorem}
\begin{proof}
The implications \ref{item:condition_lower_bound_stronger} $\Longrightarrow$ {\ref{item:condition_lower_bound_basis}} and 
{\ref{item:condition_upper_bound_stronger}} $\Longrightarrow$ 
{\ref{item:condition_upper_bound_basis}} are immediate. 

{\ref{item:ldp_lower_bound}  $\Longrightarrow$ \ref{item:condition_lower_bound_stronger}} 
Let $U\subset \cX$ be an open set. By Lemma \ref{lemma:liminf_and_limsup_of_conditional_open_compared_with_kernel_PRCP}, \eqref{eqn:conditional_on_open_giving_upper_bound}, 
\begin{align}
& \liminfn \liminf_{V\in\cV_n}
\traten \log \mu_n(U \times \cY | \cX \times V) 
\ge 
\liminfn  \traten \log \eta_n(y_n,U ).
\end{align}

{\ref{item:ldp_upper_bound}  $\Longrightarrow$ \ref{item:condition_upper_bound_stronger}} Let $W\subset \cX$ be a closed set.  
By Lemma \ref{lemma:liminf_and_limsup_of_conditional_open_compared_with_kernel_PRCP}, \eqref{eqn:conditional_on_closed_giving_lower_bound},
\begin{align}
& \limsupn \limsup_{V\in\cV_n} 
\traten \log \mu_n( W \times \cY | \cX \times V) 
\le  
\limsupn \traten \log \eta_n(y_n, W ).
\end{align}

{\ref{item:condition_lower_bound_basis} $\Longrightarrow$ \ref{item:ldp_lower_bound}}. 
 Let $G\subset \cX$ be open. 
 Let $\epsilon>0$ and $U$ be as in Lemma \ref{lemma:normal_space_small_difference_with_smaller_open_set}(a). 
Then we obtain using Lemma
\ref{lemma:liminf_and_limsup_of_conditional_open_compared_with_kernel_PRCP}
\begin{align}
\label{eqn:proof_lower_bound}
\liminfn \traten \log \eta_n(y_n,G) 
\notag &\ge \liminfn \traten \log \eta_n(y_n,\overline U) \\
\notag &\ge \liminfn \limsup_{V\in\cV_n} \traten \log \mu_n(\overline U \times \cY | \cX \times V) \\
\notag &\ge - \inf I(U) =  - \inf J (U \times \{y\}) + \inf J(\cX \times \{y\})  \\
&\ge - \inf J (G \times \{y\}) + \inf J(\cX \times \{y\}) - \epsilon.
\end{align}
As this holds for all $\epsilon>0$, we conclude  \eqref{eqn:lower_bound_eta_n_fixed_y}. 

{\ref{item:condition_upper_bound_basis} $\Longrightarrow$ \ref{item:ldp_upper_bound}}.  
Let $\alpha < \inf J ( F \times \{y\})$ and $U_1,\dots,U_k$ and $W$ be as in Lemma \ref{lemma:normal_space_small_difference_with_smaller_open_set}(b).
Then we obtain using Lemma \ref{lemma:liminf_and_limsup_of_conditional_open_compared_with_kernel_PRCP}
\begin{align}
\label{eqn:proof_upper_bound}
\limsupn \traten \log \eta_n(y_n,F) 
\notag &\le \limsupn \traten \log \eta_n(y_n, W^\circ) \\
\notag &\le \limsupn \liminf_{V\in\cV_n} \traten \log \mu_n( W^\circ \times \cY | \cX \times V) \\
 &\le - \inf I(W) \le  -\alpha + \inf J(\cX \times \{y\}). 
\end{align}
As this holds for all $\alpha < \inf J(F \times \{y\})$, we conclude \eqref{eqn:upper_bound_eta_n_fixed_y}. 
\end{proof}

\begin{obs}[Fixed $y$]
\label{obs:fixed_y}
Note that if $y_n =y$ for all $n\in\N$, one can take $\cV_n = \cV$ for a $\cV \subset \cN_y$ with $\bigcap \cV = \{y\}$. 
Then Theorem \ref{theorem:lower_and_upper_bound_eta_n_convergent_y_n} implies that $(\eta_n(y,\cdot))_{n\in\N}$ satisfies the large deviation principle with rate function $I$ if and only if {\ref{item:condition_lower_bound_basis}} and {\ref{item:condition_upper_bound_basis}} hold (with $\cV_n=\cV$). 
\end{obs}

\begin{obs}
\label{obs:non_dependence_choice_nbh_sequence_lower_condition}
Let $(y_n)_{n\in\N}$ in $\cY$  be such that $y_n \in \supp(\mu_n\circ \pi^{-1})$ and $y_n \rightarrow y$.
From Theorem \ref{theorem:lower_and_upper_bound_eta_n_convergent_y_n} we derive 
that {\ref{item:condition_lower_bound_basis}} holds for some $\cV_n\subset \cN_{y_n}$ with $\bigcap \cV_n=\{y_n\}$ if and only if {\ref{item:condition_lower_bound_basis}} holds for all such $\cV_n$. Similarly, {\ref{item:condition_upper_bound_basis}} holds for some $\cV_n\subset \cN_{y_n}$ with $\bigcap \cV_n=\{y_n\}$ if and only if {\ref{item:condition_upper_bound_basis}} holds for all such $\cV_n\subset \cN_{y_n}$.
%
\end{obs}

In Lemma \ref{lemma:limit_bounds_on_conditional_from_ldp}, we give a consequence of the large deviation principle of $(\mu_n)_{n\in\N}$. 
In Theorem \ref{theorem:equivalent_notions_bounds_prodRCP} and Theorem \ref{theorem:sufficient_bounds_for_bounds_fixed_y}
 we use this to formulate sufficient conditions for upper or lower large deviation bound on sequences $(\eta_n(y_n,\cdot))_{n\in\N}$ with  $y_n \rightarrow y$ and sequences $(\eta_n(y,\cdot))_{n\in\N}$.

We assumed $\cX$ to be normal in this section. For Lemma \ref{lemma:limit_bounds_on_conditional_from_ldp} this assumption can be dropped.

\begin{obs}
\label{obs:conditioning_eventually_welldefined_under_LDP}
For all neighbourhoods $V$ of $y$ one has by the large deviation principle
\begin{align}
\label{eqn:1nlogmu_n_G_cond_V_not_minus_infty}
& \liminfn \traten \log \mu_n( \cX \times V) \ge
 - \inf J(\cX \times V^\circ) \ge - \inf J(\cX \times \{y\}) >
-\infty.
\end{align}
In particular, there exists an $N\in \N$ such that $\mu_n (\cX \times V) >0$ for all $n\ge N$. 
Therefore $\mu_n (G \times \cY | \cX \times V)$ is well-defined for large $n$. 
\end{obs}

\begin{lemma}
\label{lemma:limit_bounds_on_conditional_from_ldp}
\
\begin{enumerate}
\item 
For open  $G\subset \cX$ 
\begin{align}
\label{eqn:lower_bound_conditioned_on_single_open_set}
\liminf_{V\in \cN_y}
\liminf_{ \putatop{ n\rightarrow \infty}{n\in \N : \mu_n(\cX \times V)>0 } } 
 \traten \log \mu_n( G \times \cY | \cX \times V) 
\ge - \inf I (G). 
\end{align}
\item 
For closed $F\subset \cX$ 
\begin{align}
\label{eqn:upper_bound_conditioned_on_single_open_set}
\limsup_{V\in \cN_y}
\limsup_{ \putatop{ n\rightarrow \infty}{n\in \N : \mu_n(\cX \times V)>0 } } 
\traten \log \mu_n( F \times \cY | \cX \times V) 
\le - \inf I(F). 
\end{align}
\end{enumerate}
\end{lemma}
\begin{proof}
(a). 
Let $\epsilon>0$. 
By Lemma \ref{lemma:convergence_infimum_rate_function_on_schrinking_neighbourhoods}, 
there exists a $V_0\in \cN_y$  
such that for all $V\in \cN_y$ with $V\subset V_0$
\begin{align}
&\inf J( \cX  \times \{y\}) \ge \inf J( \cX   \times \overline V)  \ge \inf J( \cX   \times \overline V_0) \ge \inf J( \cX  \times \{y\}) -  \epsilon. 
\label{eqn:inf_J_X_y_and_V}
\end{align}
Let $V\in \cN_y$ be such that $V \subset V_0$. 
As $\limsupn \traten \log \mu_n( \cX \times  V) > -\infty$ (see \ref{obs:conditioning_eventually_welldefined_under_LDP}) we can ``split the $\liminf$ in two'' and we get by the large deviation principle and by \eqref{eqn:inf_J_X_y_and_V}
\begin{align}
\notag 
\liminf_{ \putatop{ n\rightarrow \infty}{n\in \N : \mu_n(\cX \times V)>0 } } &  \traten \log  \mu_n  ( G \times \cY | \cX \times V) \\
\notag &
 =\liminfn \traten \log \mu_n( G \times  V) - \limsupn \traten \log \mu_n( \cX \times V) \\
& \ge - \inf J(G \times \{y\} ) + \inf J( \cX \times \overline V)
\ge - \inf I(G) - \epsilon. 
\label{eqn:liminf_open_conditional_splitted}
\end{align}

(b). Let $\alpha < \inf J(F \times \{y\})$. 
There exists a neighbourhood $V_0$ of $y$ such that for all  neighbourhoods $V$ of $y$ with $V\subset V_0$
\begin{align}
&\inf J( F  \times \{y\}) \ge \inf J( F   \times \overline V)  \ge \inf J( F   \times \overline V_0) \ge
\alpha. 
\end{align}
Let $V\in \cN_y$ be such that $y\in V \subset V_0$. Similarly as above we get
\begin{align}
\limsup_{ \putatop{ n\rightarrow \infty}{n\in \N : \mu_n(\cX \times V)>0 } }  \traten \log \mu_n( F \times \cY | \cX \times V) 
& \le - \alpha +  \inf J( \cX \times \{y\} ). 
\end{align}
\
\end{proof}

\begin{theorem}
\label{theorem:I_has_compact_sublevel_sets}
$I$ has compact sublevel sets. 
\end{theorem}
\begin{proof}
$[I\le c] = \pi( [ J \le c +\inf J(\cX \times \{y\})])$. 
\end{proof}

\begin{theorem}
\label{theorem:equivalent_notions_bounds_prodRCP}
We have 
\begin{center}
\ref{item:compared_condition_lower_bound_without_limit} $\Longrightarrow$
\ref{item:compared_condition_lower_bound_logical} $\iff$ \ref{item:compared_condition_lower_bound} $\Longrightarrow$ \ref{item:uniform_condition_lower_bound} $\Longrightarrow$ \ref{item:forall_y_seq_cond_lower_bound}, 
\end{center}
and, if $\cY$ is first countable, then
\begin{center}
 \ref{item:forall_y_seq_cond_lower_bound}
$\iff$ \ref{item:uniform_condition_lower_bound},
\end{center}
where
\begin{enumerate}[label={\normalfont(A\arabic*)},topsep=4pt,itemsep=0pt]
\item 
\label{item:forall_y_seq_cond_lower_bound}
For all $(y_n)_{n\in\N}$ with $y_n \in \supp(\mu_n\circ \pi^{-1})$ and $y_n \rightarrow y$ 
the sequence $(\eta_n(y_n,\cdot))_{n\in\N}$ satisfies the large deviation lower bound with rate function $I$. 
\item 
\label{item:uniform_condition_lower_bound}
For all  $U\in \cG$ 
\begin{align}
\sup_{V_0\in \cN_y} \liminfn \inf_{ \putatop{V\in \cH, V \subset V_0}{V\cap \supp(\mu_n \circ \pi^{-1})\ne \emptyset} }
\traten \log \mu_n( \overline  U \times \cY | \cX \times V) 
 \ge - \inf I(U). 
\end{align}
\item 
\label{item:compared_condition_lower_bound}
For all $U\in \cG$ 
\begin{align}
\notag & \sup_{V_0\in \cN_y} \liminfn \inf_{ \putatop{V\in \cH, V \subset V_0}{V\cap \supp(\mu_n \circ \pi^{-1})\ne \emptyset} }
\traten \log \mu_n( \overline  U \times \cY | \cX \times V) \\
& \ge 
\liminf_{V\in \cN_y}
\liminf_{ \putatop{ n\rightarrow \infty}{n\in \N : \mu_n(\cX \times V)>0 } } 
 \traten \log \mu_n(  U \times \cY | \cX \times V).
 \label{eqn:sufficient_condition_lower_bound_comparison}
\end{align}
\item 
\label{item:compared_condition_lower_bound_logical}
 For all $U\in \cG$ 
 we have 
$\forall Z_0\in \cN_y \forall \epsilon>0 \exists V_0\in \cN_y \exists Z\in \cN_y, Z\subset Z_0 \forall M \exists m \ge M \exists N \forall n\ge N \forall V\in \cH, V \subset V_0, V\cap \supp(\mu_n \circ \pi^{-1})\ne \emptyset $:
\begin{align}
\traten 
\log \mu_n ( \overline{U} \times \cY |  \cX \times V) 
\ge 
\tratem 
\log \mu_m ( U \times \cY | \cX \times Z) - \epsilon. 
\end{align}
\item 
\label{item:compared_condition_lower_bound_without_limit}
For all $U\in \cG$ 
we have $\forall \epsilon>0 \forall V_0 \in \cN_y \exists N\in\N \forall n \ge N \forall V\in \cH, V\subset V_0, V\cap \supp(\mu_n \circ \pi^{-1})\ne \emptyset$:
\begin{align}
\label{eqn:finite_n_lower_bound_without_log}
\mu_n(  \overline U  \times \cY | \cX \times V)
\ge 
e^{-n\epsilon} \mu_n(  U  \times \cY | \cX \times V_0).
\end{align}
\end{enumerate}
Moreover,
\begin{center}
\ref{item:compared_condition_upper_bound_without_limit} $\Longrightarrow$
\ref{item:compared_condition_upper_bound_logical} $\iff$
 \ref{item:compared_condition_upper_bound} $\Longrightarrow$ \ref{item:uniform_condition_upper_bound} $\Longrightarrow$ \ref{item:forall_y_seq_cond_upper_bound}, 
\end{center}
and, if $\cY$ is first countable, then
\begin{center}
\ref{item:forall_y_seq_cond_upper_bound}
$\iff$ \ref{item:uniform_condition_upper_bound},
\end{center}
where
\begin{enumerate}[label={\normalfont(B\arabic*)},topsep=4pt,itemsep=0pt]
\item 
\label{item:forall_y_seq_cond_upper_bound}
For all $(y_n)_{n\in\N}$ with $y_n \in \supp(\mu_n\circ \pi^{-1})$ and $y_n \rightarrow y$
the sequence $(\eta_n(y_n,\cdot))_{n\in\N}$ satisfies the large deviation upper bound with rate function $I$. 
\item 
\label{item:uniform_condition_upper_bound}
For all  $U_1,\dots,U_k\in \cG$ one has for $W= \cX \setminus ( U_1 \cup \cdots \cup U_k)$
\begin{align}
 & \inf_{V_0\in \cN_y} \limsupn \sup_{ \putatop{V\in \cH, V \subset V_0}{V\cap \supp(\mu_n \circ \pi^{-1})\ne \emptyset} }
\traten \log \mu_n( W^\circ \times \cY | \cX \times V) \le - \inf I(W). 
\end{align}
\item 
\label{item:compared_condition_upper_bound}
For all $U_1,\dots, U_k\in \cG$ with $W= \cX \setminus (U_1\cup \cdots \cup U_k)$
\begin{align}
\notag &  \inf_{V_0\in \cN_y} \limsupn \sup_{ \putatop{V\in \cH, V \subset V_0}{V\cap \supp(\mu_n \circ \pi^{-1})\ne \emptyset} }
\traten \log \mu_n( W^\circ \times \cY | \cX \times V)  \\ 
& \le 
\limsup_{V\in\cN_y}
\limsup_{ \putatop{ n\rightarrow \infty}{n\in \N : \mu_n(\cX \times V)>0 } } 
\traten \log \mu_n( W \times \cY | \cX \times V). 
\label{eqn:sufficient_condition_upper_bound_comparison}
\end{align}
\item 
\label{item:compared_condition_upper_bound_logical}
For all $U_1,\dots, U_k\in \cG$ with $W= \cX \setminus (U_1\cup \cdots \cup U_k)$ we have 
$\forall Z_0\in \cN_y \forall \epsilon>0 \exists V_0\in \cN_y \exists Z\in \cN_y, Z\subset Z_0 \forall M \exists m \ge M \exists N \forall n\ge N \forall V\in \cH, V \subset V_0,V\cap \supp(\mu_n \circ \pi^{-1})\ne \emptyset$:
\begin{align}
\traten 
\log \mu_n ( \overline{U} \times \cY |  \cX \times V) 
\le  
\tratem 
\log \mu_m ( U \times \cY | \cX \times Z) + \epsilon. 
\end{align}

\item 
\label{item:compared_condition_upper_bound_without_limit}
For all $U_1,\dots, U_k\in \cG$ with $W= \cX \setminus (U_1\cup \cdots \cup U_k)$ we have $\forall \epsilon>0 \forall V_0 \in \cN_y \exists N\in\N \forall n \ge N \forall V\in \cH, V\subset V_0,V\cap \supp(\mu_n \circ \pi^{-1})\ne \emptyset$:
\begin{align}
\label{eqn:finite_n_upper_bound_without_log}
\mu_n( W^\circ  \times \cY | \cX \times V)
\le 
e^{n\epsilon} \mu_n( W \times \cY | \cX \times V_0)
\end{align}

\end{enumerate}
\end{theorem}
\begin{proof}
The proofs of 
{\ref{item:compared_condition_upper_bound_without_limit}} $\Longrightarrow$
{\ref{item:compared_condition_upper_bound_logical}} $\iff$
{\ref{item:compared_condition_upper_bound}} $\Longrightarrow$ {\ref{item:uniform_condition_upper_bound}} $\Longrightarrow$ {\ref{item:forall_y_seq_cond_upper_bound}} and of 
{\ref{item:forall_y_seq_cond_upper_bound}}
$\Longrightarrow$ 
{\ref{item:uniform_condition_upper_bound}} are similar to the proofs of the following implications. 

{\ref{item:compared_condition_lower_bound_logical}} $\iff$ {\ref{item:compared_condition_lower_bound}} follows by definition of $\sup$, $\inf$, $\limsup$ and $\liminf$. 

{\ref{item:compared_condition_lower_bound_without_limit}}
$\Longrightarrow$
{\ref{item:compared_condition_lower_bound}} 
Let $U \in \cG$. 
Assuming {\ref{item:compared_condition_lower_bound_without_limit}} we obtain 
 $\forall \epsilon>0 \forall V_0 \in \cN_y 
  \exists N\in\N  \forall n\ge N$ one has $\mu_n(\cX \times V_0)>0$ and
\begin{align}
 \inf_{ \putatop{V\in \cH, V\subset V_0 }{ V\cap  \supp(\mu_n \circ \pi^{-1})\ne \emptyset} } 
 \traten \log \mu_n( \overline U  \times \cY | \cX \times V)
\ge 
\traten \log \mu_n(  U  \times \cY | \cX \times V_0) - \epsilon.
\end{align}
So $\forall \epsilon>0 \forall V_0 \in \cN_y$ 
\begin{align}
\notag  \sup_{Z\in \cN_y} \liminfn
&  \inf_{ \putatop{V\in \cH, V\subset Z }{ V\cap  \supp(\mu_n \circ \pi^{-1})\ne \emptyset} } 
 \traten \log \mu_n( \overline U  \times \cY | \cX \times V) \\
& \ge 
\liminf_{ \putatop{ n\rightarrow \infty}{n\in \N : \mu_n(\cX \times V_0)>0 } }  
\traten \log \mu_n(  U  \times \cY | \cX \times V_0) - \epsilon.
\end{align}

{\ref{item:compared_condition_lower_bound}}
$\Longrightarrow$
{\ref{item:uniform_condition_lower_bound}}
Follows by Lemma \ref{lemma:limit_bounds_on_conditional_from_ldp}.

{\ref{item:uniform_condition_lower_bound}}
$\Longrightarrow$
{\ref{item:forall_y_seq_cond_lower_bound}}. 
Suppose that {\ref{item:uniform_condition_lower_bound}} holds. 
Let $U\in \cG$ with $\inf J(U \times \{y\})<\infty$ and let $\epsilon>0$. 
Let $V_0 \in \cN_y$ and $N\in\N$ be such that $\traten \log \mu_n( \overline U \times \cY | \cX \times V) \ge - \inf I(U) - \epsilon$ for all $n\ge N$ and all $V\in \cH$ with $V\subset V_0$ and $V\cap \supp(\mu_n \circ \pi^{-1})\ne \emptyset$. 
Let $(y_n)_{n\in\N}$ be such that $y_n \in \supp(\mu_n\circ \pi^{-1})$ and $y_n \rightarrow y$. 
Let $N_0\ge N$ be such that $y_n \in V_0$ for all $n\ge N_0$. 
Then for all $n\ge N_0$ and $V \in \cN_{y_n}\cap \cH$ with $V\subset V_0$ we have $\traten \log \mu_n( \overline U \times \cY | \cX \times V) \ge - \inf I(U) - \epsilon$. This implies {\ref{item:condition_lower_bound_basis}} of Theorem \ref{theorem:lower_and_upper_bound_eta_n_convergent_y_n} (with $\cV_n=\cN_{y_n}\cap \cH$). 

{\ref{item:forall_y_seq_cond_lower_bound}}
$\Longrightarrow$
{\ref{item:uniform_condition_lower_bound}} (assuming $\cY$ is first countable). 
Suppose that {\ref{item:uniform_condition_lower_bound}} does not hold. 
Let $(V_m)_{m\in\N}$ be a decreasing sequence in $\cH$ with $\bigcap_{m\in\N} V_m = \{y\}$. 
Then there exists a $U\in \cG$ with $\inf J(U \times \{y\})<\infty$ and an $\alpha >\inf I(U)$  such that for all $M\in\N$ and $N\in\N$ there exists an $n\ge N$ and a $V\in \cH$ with $V\subset V_M$ and $V \cap \supp( \mu_n \circ \pi^{-1})\ne \emptyset$ such that 
\begin{align}
\traten \log \mu_n( \overline U \times \cY | \cX \times V) 
\le - \alpha . 
\end{align}
Let $\beta<\alpha$ be such that $\beta > \inf I(U)$. 
By Lemma \ref{lemma:liminf_and_limsup_of_conditional_open_compared_with_kernel_PRCP} we have 
\begin{align}
\inf_{z\in V \cap \supp(\mu_n \circ \pi^{-1})} \traten \log  \eta_n(z,U) \le \traten \log \mu_n( \overline U \times \cY | \cX \times V).
\end{align}
For each $m\in\N$ there exists an $n_m$ and a $y_{n_m}\in V_m \cap \supp( \mu_{n_m} \circ \pi^{-1})$ such that 
\begin{align}
\tratenm \log \eta_{n_m}( y_{n_m}, U) 
\le -  \beta. 
\end{align}
We may choose $n_1<n_2<n_3<\cdots$. With $y_k = y$ for $k\notin \{n_m: m\in\N\}$ we have $y_n \rightarrow y$ and 
\begin{align}
 \liminfn \traten \log \eta_n(y_n,U) \le \liminfm \tratenm \log \eta_{n_m} (y_{n_m},U) \le
-  \beta. 
\end{align}
Therefore {\ref{item:ldp_lower_bound}} of Theorem \ref{theorem:lower_and_upper_bound_eta_n_convergent_y_n} does not hold, which implies that {\ref{item:forall_y_seq_cond_lower_bound}} does not hold.
\end{proof}

We can also use Lemma \ref{lemma:limit_bounds_on_conditional_from_ldp} and Theorem \ref{theorem:lower_and_upper_bound_eta_n_convergent_y_n} (see also \ref{obs:fixed_y}) to obtain sufficient conditions for the lower or upper large deviation bounds for $(\eta_n(y,\cdot))_{n\in\N}$.

\begin{theorem}
\label{theorem:sufficient_bounds_for_bounds_fixed_y}
Let $\cV\subset \cN_y$ be such that $\bigcap \cV = \{y\}$.
\begin{enumerate}
\item Suppose that for all $U\in \cG$ with $\inf J(U\times \{y\})<\infty$
\begin{align}
\notag & \liminfn 
\limsup_{V\in\cV}
\traten \log \mu_n( \overline U \times \cY | \cX \times V)   \\
& \ge 
\liminf_{V\in \cN_y}
\liminf_{ \putatop{ n\rightarrow \infty}{n\in \N : \mu_n(\cX \times V)>0 } } 
 \traten \log \mu_n(  U \times \cY | \cX \times V).
 \label{eqn:sufficient_condition_lower_bound_comparison_fixed_y}
\end{align}
Then $(\eta_n(y,\cdot))_{n\in\N}$ satisfies the large deviation lower bound with rate function $I$. 
\item Suppose that for all $U_1,\dots, U_k\in \cG$ with $W= \cX \setminus (U_1\cup \cdots \cup U_k)$
\begin{align}
\notag &  \limsupn 
\liminf_{V\in \cV}
\traten \log \mu_n( W^\circ \times \cY | \cX \times V)  \\
& \le 
\limsup_{V\in\cN_y}
\limsup_{ \putatop{ n\rightarrow \infty}{n\in \N : \mu_n(\cX \times V)>0 } } 
\traten \log \mu_n( W \times \cY | \cX \times V). 
\label{eqn:sufficient_condition_upper_bound_comparison_fixed_y}
\end{align}
Then $(\eta_n(y,\cdot))_{n\in\N}$ satisfies the large deviation upper bound with rate function $I$. 
\end{enumerate}
\end{theorem}

\begin{obs}
\eqref{eqn:sufficient_condition_lower_bound_comparison_fixed_y} and 
 \eqref{eqn:sufficient_condition_upper_bound_comparison_fixed_y} hold for example when
$\forall \epsilon>0 \forall V_0 \in \cV 
  \exists N\in\N  \forall n\ge N \forall V\in \cV, V\subset V_0: $ 
\begin{align}
& \mu_n( \overline U  \times \cY | \cX \times V)
\ge 
e^{-n\epsilon}  \mu_n( \overline U  \times \cY | \cX \times V_0), \\
& \mu_n( W^\circ  \times \cY | \cX \times V)
\le 
e^{n\epsilon} \mu_n( W  \times \cY | \cX \times V_0), 
\end{align}
respectively. 
\end{obs}

\begin{obs}
\label{obs:deriving_the_main_theorems}
Theorem \ref{theorem:equivalent_notions_ldp_bounds_metric_PRCP} is a consequence of Theorem \ref{theorem:weak_convergence_equivalences}, Theorem \ref{theorem:I_has_compact_sublevel_sets} and Theorem \ref{theorem:equivalent_notions_bounds_prodRCP} with  $\cG = \{ B(x,r): x\in \cX, r>0\}$ and $\cH = \{B(y,\delta): y\in \cY, \delta>0\}$. 
\end{obs}

\section{Large deviations for regular conditional probabilities}
\label{section:ldp_RCP}

\assump{In this section $\cX$ and $\cY$ are topological spaces, 
$(\nu_n)_{n\in\N}$ is a sequence of probability measures on $\cB(\cX)$ that satisfies the large deviation principle with rate function $\K: \cX \rightarrow [0,\infty]$ and  $\tau: \cX \rightarrow \cY$ is continuous. For more assumptions, see \ref{obs:more_assumptions}.}

We derive the analogues statements as in Section \ref{section:ldp_product_RCP} but for regular conditional kernels instead of product regular conditional kernels (\ref{obs:RCP_in_same_situation_as_PRCP} and Theorem \ref{theorem:equivalent_notions_bounds_RCP}).
First we show that with $\mu_n$ the probability measure corresponding on the product space corresponding to $\nu_n$ as in Theorem \ref{theorem:extension_and_RCP_as_special_case_of_PRCP}, the sequence $(\mu_n)_{n\in\N}$ satisfies the large deviation principle with a rate function described in terms of $K$ (Theorem \ref{theorem:ldp_for_extension_measure}). 


If $(\eta_n)_{n\in\N}$ are regular conditional probabilities under $(\nu_n)_{n\in\N}$ given $\tau$, then
one could also follow the proofs in 
Section \ref{section:ldp_product_RCP} 
for the product regular conditional probabilities 
to obtain similar results for large deviations for sequences of the form $(\eta_n(y_n,\cdot))_{n\in\N}$. 
Instead, we make the approach via Theorem \ref{theorem:extension_and_RCP_as_special_case_of_PRCP} to translate the results to the setting of regular conditional probabilities. 

\begin{theorem}
\label{theorem:ldp_for_extension_measure}
For all $n\in\N$ let $\mu_n$ be the probability measure on $ \cB(\cX) \otimes \cB(\cY)$ for which $\mu_n(A\times B) = \nu_n ( A \cap \tau^{-1}(B))$ for $A\in \cB(\cX), B\in \cB(\cY)$ (as in Theorem \ref{theorem:extension_and_RCP_as_special_case_of_PRCP}). 
Then $(\mu_n)_{n\in\N}$ satisfies the large deviation principle on $\{A\times B: A\in \cB(\cX), B\in  \cB(\cY)\}$ 
 with rate function $\J: \cX \times \cY \rightarrow [0,\infty]$ given by
\begin{align}
\J(x, y) = 
\begin{cases}
\K(x) &  \tau(x) = y, \\
\infty 	& \tau(x) \ne y. 
\end{cases}
\end{align}
If $\K$ has compact sublevel sets, then so does $\J$. 
\end{theorem}
\begin{proof}
By definition of $\J$ we have 
\begin{align}
\label{eqn:relation_K_and_J}
\inf \K(A \cap \tau^{-1}(B))= \inf \J(A\times B) \qquad \big( A\in \cB(\cX) , B\in \cB(\cY)\big). 
\end{align}
Let $A\in \cB(\cX)$ and $B\in \cB(\cY)$. Then
\begin{align}
\liminfn \traten \log \mu_n( A\times B)
\notag & =\liminfn \traten \log \nu_n (A\cap \tau^{-1}(B)) \\
& \ge - \inf \K( (A\cap \tau^{-1}(B))^\circ).
\end{align}
We have $(A \cap \tau^{-1}(B))^\circ= A^\circ \cap \tau^{-1}(B)^\circ$ and $ \tau^{-1}(B)^\circ \supset \tau^{-1}(B^\circ)$, whence 
\begin{align}
\inf \K( (A\cap \tau^{-1}(B))^\circ) 
\notag & \le \inf \K( A^\circ \cap \tau^{-1}(B^\circ)) \\
& = \inf \J( A^\circ \times B^\circ) = \inf \J ( (A\times B)^\circ).
\end{align}
Similarly
\begin{align}
\limsupn \traten \log \mu_n( A\times B) 
\notag & = \limsupn \traten \log \nu_n( A \cap \tau^{-1}(B)) \\
& \le - \inf \K ( \overline{A \cap \tau^{-1}(B)}). 
\end{align}
We have $\overline{A \cap \tau^{-1}(B)} \subset \overline{A} \cap \overline{\tau^{-1}(B)}$ and 
$\overline{\tau^{-1}(B)} \subset \tau^{-1}(\overline B)$, whence 
\begin{align}
\inf \K ( \overline{ A \cap \tau^{-1}(B)}) \ge \inf \K(\overline A \cap \tau^{-1}(\overline{B})) = \inf \J( \overline{A \times  B}).
\end{align}
Suppose that $\K$ has compact sublevel sets. 
Let $c\ge 0$. 
Then $[J\le c]$ is contained in the compact set $[\K\le c] \times \tau([\K \le c])$. 
By Theorem \ref{theorem:I_has_compact_sublevel_sets} $I$ has compact sublevel sets. 
\end{proof}

\begin{obs}
\label{obs:more_assumptions}
\assump{In the rest of this section $\cX$ is normal, 
$\cG$, $\cH$, $\pi$ are as in \ref{item:basis_topologies} and \ref{item:pi_map} of Section \ref{section:ldp_product_RCP}.
Furthermore similarly to
\ref{item:existence_prcp_mun_wrt_pi} and
\ref{item:y_finiteness_rate_function_and_def_I} of Section \ref{section:ldp_product_RCP} we assume the following. 
\begin{enumerate}[label=(\roman*)]
\setitemize{leftmargin=0pt}
\setlength{\itemsep}{0pt}
  \setlength{\parskip}{0pt}
  \setlength{\parsep}{0pt}
\item[\ref{item:existence_prcp_mun_wrt_pi}*] 
For each $n\in\N$ we assume the following: 
$\supp ( \nu_n \circ \tau^{-1}) \ne \emptyset$,
there exists a regular conditional probability  $\eta_n : \cY \times \cB(\cX) \rightarrow [0,1]$ under $\nu_n$ with respect to $\tau$, satisfying the continuity condition \eqref{eqn:continuity_assumption}.
\item[\ref{item:y_finiteness_rate_function_and_def_I}*] Let $y \in \cY$. We assume that $\inf \K (\tau^{-1}(\{y\}))<\infty$  and that there exist $y_n \in \supp( \nu_n \circ \tau^{-1})$ with $y_n \rightarrow y$. 
Let $I: \cX \rightarrow [0,\infty]$ be given by 
\begin{align}
I(x) 
\notag & = \J(x,y) - \inf \J(\cX \times \{y\})  \\
& = 
\begin{cases}
 \K(x) - \inf \K(\tau^{-1}(\{y\})) & \tau(x)=y, \\
 \infty & \tau(x) \ne y. 
\end{cases}
\end{align}
\end{enumerate}
}
\end{obs}

\begin{obs}
\label{obs:RCP_in_same_situation_as_PRCP}
As by Theorem \ref{theorem:extension_and_RCP_as_special_case_of_PRCP} $\eta_n$ is the product regular conditional kernel under $\mu_n$ with respect to $\pi$, by Theorem \ref{theorem:ldp_for_extension_measure} $(\mu_n)_{n\in\N}$ 
satisfies the large deviation principle on $\{A\times B: A\in \cB(\cX), B\in \cB(\cY)\}$ with rate function $J$, and $\inf J(\cX \times \{y\}) = \inf K(\tau^{-1}(y))<\infty$ and 
$\mu_n$ and $\eta_n$ are as in Section \ref{section:ldp_product_RCP} (in the sense that \ref{item:ass_ldp_mu_n},
\ref{item:existence_prcp_mun_wrt_pi},
\ref{item:y_finiteness_rate_function_and_def_I} hold).
Therefore we can translate the results of Section \ref{section:ldp_product_RCP}, but also the results of Section \ref{section:weak_continuous_cps} and Section \ref{section:facts_about_lsc_functions_with_compact_level_sets}, using for example \eqref{eqn:relation_K_and_J}, $ \nu_n \circ \tau^{-1} = \mu_n \circ \pi^{-1}$ and that for $V\in \cB(\cY)$ with $\nu_n(\tau^{-1}(V))>0$ and for $A\in \cB(\cX)$
\begin{align}
& \nu_n( A | \tau^{-1}(V) ) = \mu_n(A \times \cX | \cX \times V).
\end{align}
In this sense also Theorem \ref{theorem:equivalent_notions_ldp_bounds_metric_RCP} follows from Theorem \ref{theorem:equivalent_notions_ldp_bounds_metric_PRCP}.
We present some of the equivalent statements of Theorem \ref{theorem:equivalent_notions_bounds_prodRCP} in 
Theorem \ref{theorem:equivalent_notions_bounds_RCP}.
\end{obs}

\begin{remark}
\label{remark:motivation_LDP_definition}
Because of the relation between $\mu_n$ and $\nu_n$ and between $K$ and $J$, in Theorem \ref{theorem:ldp_for_extension_measure} we were able to prove the large deviation principle on $\{A\times B: A\in \cB(\cX), B\in \cB(\cY)\}$. 
Whether it can be extended to the large deviation principle on $\cB(\cX)\otimes \cB(\cY)$ is a priori not clear. 
However, for the purpose of using the results of Section \ref{section:ldp_product_RCP} this is not required (as only \ref{item:ass_ldp_mu_n} of Section \ref{section:ldp_product_RCP} is required). This is the main reason to define the large deviation bounds as in Definition \ref{def:ldp_adapted}.
\end{remark}

\begin{theorem}
\label{theorem:equivalent_notions_bounds_RCP}
 \ref{item:compared_condition_lower_bound} $\Longrightarrow$ \ref{item:uniform_condition_lower_bound} $\Longrightarrow$ \ref{item:forall_y_seq_cond_lower_bound}. If $\cY$ is first countable, then
 \ref{item:forall_y_seq_cond_lower_bound}
$\iff$ \ref{item:uniform_condition_lower_bound}. 
\begin{enumerate}[label={\normalfont(A\arabic*)},topsep=4pt,itemsep=0pt] 
\item 
\label{item:forall_y_seq_cond_lower_bound_RCP}
For all $(y_n)_{n\in\N}$ with $y_n \in \supp(\nu_n\circ \tau^{-1})$ and $y_n \rightarrow y$ 
the sequence $(\eta_n(y_n,\cdot))_{n\in\N}$ satisfies the large deviation lower bound with rate function $I$. 
\item 
\label{item:uniform_condition_lower_bound_RCP}
For all  $U\in \cG$ 
\begin{align}
\sup_{V_0\in \cN_y} \liminfn \inf_{ \putatop{V\in \cH, V \subset V_0}{V\cap \supp(\nu_n \circ \tau^{-1})\ne \emptyset} }
\traten \log \nu_n( \overline  U  | \tau^{-1}( V) )
 \ge - \inf I(U). 
\end{align}
\item 
\label{item:compared_condition_lower_bound_RCP}
For all $U\in \cG$ 
\begin{align}
\notag & \sup_{V_0\in \cN_y} \liminfn \inf_{ \putatop{V\in \cH, V \subset V_0}{V\cap \supp(\nu_n \circ \tau^{-1})\ne \emptyset} }
\traten \log \nu_n( \overline  U  | \tau^{-1}(V) ) \\
& \ge 
\liminf_{V\in \cN_y}
\liminf_{ \putatop{ n\rightarrow \infty}{n\in \N : \nu_n( \tau^{-1}( V))>0 } } 
 \traten \log \nu_n(  U | \tau^{-1}( V)) .
 \label{eqn:sufficient_condition_lower_bound_comparison_RCP}
\end{align}
\end{enumerate}
 \ref{item:compared_condition_upper_bound} $\Longrightarrow$ \ref{item:uniform_condition_upper_bound} $\Longrightarrow$ \ref{item:forall_y_seq_cond_upper_bound}. If $\cY$ is first countable then
 \ref{item:forall_y_seq_cond_upper_bound}
$\iff$ \ref{item:uniform_condition_upper_bound}.

\begin{enumerate}[label={\normalfont(B\arabic*)},topsep=4pt,itemsep=0pt]
\item 
\label{item:forall_y_seq_cond_upper_bound_RCP}
For all $(y_n)_{n\in\N}$ with $y_n \in \supp(\nu_n\circ \tau^{-1})$ and $y_n \rightarrow y$
the sequence $(\eta_n(y_n,\cdot))_{n\in\N}$ satisfies the large deviation upper bound with rate function $I$. 
\item 
\label{item:uniform_condition_upper_bound_RCP}
For all  $U_1,\dots,U_k\in \cG$ one has for $W= \cX \setminus ( U_1 \cup \cdots \cup U_k)$
\begin{align}
 & \inf_{V_0\in \cN_y} \limsupn \sup_{ \putatop{V\in \cH, V \subset V_0}{V\cap \supp(\nu_n \circ \tau^{-1})\ne \emptyset} }
\traten \log \nu_n( W^\circ | \tau^{-1}( V ) ) \le - \inf I(W). 
\end{align}
\item 
\label{item:compared_condition_upper_bound_RCP}
For all $U_1,\dots, U_k\in \cG$ with $W= \cX \setminus (U_1\cup \cdots \cup U_k)$
\begin{align}
\notag &  \inf_{V_0\in \cN_y} \limsupn \sup_{ \putatop{V\in \cH, V \subset V_0}{V\cap \supp(\nu_n \circ \tau^{-1})\ne \emptyset} }
\traten \log \nu_n( W^\circ  | \tau^{-1}( V ) )  \\ 
& \le 
\limsup_{V\in\cN_y}
\limsup_{ \putatop{ n\rightarrow \infty}{n\in \N : \nu_n( \tau^{-1}( V))>0 } } 
\traten \log \nu_n( W | \tau^{-1}( V ) ). 
\label{eqn:sufficient_condition_upper_bound_comparison_RCP}
\end{align}
\end{enumerate}
\end{theorem}

\section{An application to conditional probabilities of empirical distributions on finite sets}
\sectionmark{An application to conditional probabilities of empirical distributions}
\label{section:application_sanov_type}

In terms of random variables, Sanov's Theorem gives us the large deviation principle of empirical densities $\frac1n\sum_{i=1}^n \delta_{X_i}$, where $X_1,X_2,\dots$ are independent and identically distributed random variables. 
We consider large deviations of $\frac1n\sum_{i=1}^n \delta_{X_i}$ conditioning on $\frac1n\sum_{i=1}^n \delta_{Y_i} = \psi_n$, where $(X_1,Y_1), (X_2,Y_2), \dots$ are independent and identically distributed couples of random variables, both random variables attaining their values in a finite set. This large deviation principle is formalised in Theorem \ref{theorem:application_sanov_type}.

\bigskip 

\assump{In this section we consider the following.} 
\begin{itemize}
\setitemize{leftmargin=0pt}
\setlength{\itemsep}{0pt}
  \setlength{\parskip}{0pt}
  \setlength{\parsep}{0pt}
\item Let $\cR$ and $\cS$ be finite sets equipped with the discrete topology (discrete metric). Let $\cP(\cR),\cP(\cS)$ and $\cP(\cR \times \cS)$ be equipped by the weak topology and let $\fd$ denote the Prohorov metric (see Billingsley \cite[Appendix III]{Bi68}) on each of the spaces. 
\item Let $\lambda \in \cP(\cR \times \cS)$. We assume $\lambda( \cR \times \{s\}) >0$ for all $s\in \cS$. 
\item For $n\in\N$ let $L_n : \cR^n \rightarrow \cP(\cR)$ be given by $L_n(r) = \frac1n \sum_{i=1}^n \delta_{r_i}$ 
for $r=(r_1,\dots,r_n)\in \cR^n$. 
\item Write $\cP_{emp}^n(\cR) = L_n(\cR^n) = \{ \frac1n \sum_{i=1}^n \delta_{r_i} : r_1,\dots,r_n\in\cR\}$, similarly $\cP_{emp}^n(\cS)= L_n(\cS^n)$ and $\cP_{emp}^n(\cR\times \cS)= L_n((\cR\times \cS)^n)$.
\item Let $\margi : \cP(\cR \times \cS) \rightarrow \cP(\cR) \times \cP(\cS)$ be the map that maps a measure in $\cP(\cR \times \cS)$ onto the pair of its marginals, i.e., $\margi$ is given by 
\begin{align}
\margi( \xi) = \big(  \xi( \cdot \times \cS), \xi(\cR \times \cdot) \big).
\end{align}
\item Let $\pi : \cP(\cR) \times \cP(\cS) \rightarrow \cP(\cS)$ be the map given by $\pi(\xi,\zeta) = \zeta$. 
\item Let $\mu_n$ be the probability measure on $\cB(\cP(\cR)) \otimes \cB(\cP(\cS))$ defined by \\
$\mu_n = \left( \bigotimes_{i=1}^n \lambda \right) \circ L_n^{-1} \circ \margi^{-1}$, so that for $A\in \cB(\cP(\cR))$ and $B\in \cB(\cP(\cS))$
\begin{align}
\mu_n( A \times B) 
= \left( \bigotimes_{i=1}^n \lambda \right) ( L_n^{-1}(A) \times L_n^{-1}(B)). 
\end{align}
\item Define $\theta : \cS \times \cB(\cR) \rightarrow [0,1]$ by $\theta(s,A) = \lambda( A\times \cS | \cR \times \{s\})$. 
\item Define $\eta_n : \cP(\cS) \times \cB(\cP(\cR)) \rightarrow [0,1]$ by 
\begin{align}
\label{eqn:eta_n_for_sanov_type}
\eta_n(\xi,A) = 
\begin{cases}
\left[ \bigotimes_{i=1}^n \theta(s_i,\cdot) \right] \circ L_n^{-1}(A) &  \xi \in \cP_{emp}^n(\cS), \xi = L_n(s_1,\dots,s_n) \\
&  \mbox{for }  s_1,\dots,s_n \in \cS, \\
0 & \xi \notin \cP_{emp}^n(\cS). 
\end{cases}
\end{align}
\item Let $J: \cP(\cR) \times \cP(\cS) \rightarrow [0,\infty]$ be given by 
\begin{align}
J(\rho, \sigma) 
 & = \inf_{\xi \in \margi^{-1} (\{ (\rho,\sigma) \}) } H(\xi | \lambda ).
\end{align}
where $H(\xi | \lambda)$ is the relative entropy of $\xi$ with respect to $\lambda$ (\cite[Definition 2.1.5]{DeZe10}).
\item Let $\psi \in \cP(\cS)$ be such that 
\begin{align}
\inf_{ \xi \in \margi^{-1} (\cP(\cR) \times \{\psi\}) } H( \xi |\lambda)<\infty. 
\end{align}
\end{itemize}

\begin{obs}
\label{obs:some_fact_finite_sanov_case}
We present some fact which follow from the assumptions with little effort; to some facts we give some explanation or references. 
\begin{enumerate}
\item \label{item:closedness_of_empirical_measures_in_space_of_measures}
$\cP_{emp}^n(\cS)$ is closed in $\cP(\cS)$. 
Moreover, if $\xi_k$ and $\xi$ in $\cP_{emp}^n(\cS)$ are such that $\xi_k \rightarrow \xi$, then there exist $s_{ki}$ and $q_i$ in $\cS$ for $i\in \{1,\dots,n\}$ such that $\xi_k = L_n((s_{k1},\dots,s_{kn}))$, $\xi = L_n( (q_1,\dots,q_n))$ and $s_{ki} \rightarrow q_i$ for all $i\in \{1,\dots,n\}$.
\item \label{item:support_is_empirical_measures}
$\supp (\mu_n \circ \pi^{-1}) = \cP_{emp}^n(\cS)$. 
\item \label{item:eta_n_is_rck_and_continuous}
$\eta_n$ is a product regular conditional kernel under $\mu_n$ with respect to $\pi$ that is weakly continuous on $\cP_{emp}^n(\cS)$. 
\item \label{item:sanov_for_finite}
$( \bigotimes \lambda^n \circ L_n^{-1})_{n\in\N}$ satisfies the large deviation principle with rate function $ H( \cdot |\lambda)$. 
\item \label{item:margi_is_continuous}
$\margi$ is continuous. 
\item \label{item:mu_n_satisfies_ldp}
$(\mu_n)_{n\in\N}$ satisfies the large deviation principle with rate function $J$. 
\end{enumerate}
\ref{item:closedness_of_empirical_measures_in_space_of_measures} follows from the fact that $\cS$ is a finite space. 
\ref{item:support_is_empirical_measures} follows from \ref{item:closedness_of_empirical_measures_in_space_of_measures}, from the fact that the complement of $\cP_{emp}^n(\cS)$ has $\mu_n \circ \pi^{-1}$-measure zero and because 
$\mu_n \circ \pi^{-1}(\{L_n(s)\}) >0 $ for all $s\in \cS^n$, which is due to the assumptions on $\lambda$. 
\ref{item:eta_n_is_rck_and_continuous} follows by a straightforward calculation, the continuity follows from \ref{item:closedness_of_empirical_measures_in_space_of_measures}. 
For \ref{item:sanov_for_finite} see Sanov's Theorem (Dembo and Zeitouni \cite[Theorem 6.2.10]{DeZe10}).   
\ref{item:margi_is_continuous} follows from the fact that if $\xi_n \rightarrow \xi$ in $\cP(\cR \times \cS)$, then the $\cR$- and $\cS$-marginals of $\xi_n$ converge to the $\cR$- and $\cS$-marginals of $\xi$, respectively. 
Then \ref{item:mu_n_satisfies_ldp} follows from \ref{item:margi_is_continuous} and \ref{item:sanov_for_finite} by the contraction principle \cite[Theorem 4.2.1]{DeZe10}. 
\end{obs}

In the rest of this section we prove the following theorem. 

\begin{theorem}
\label{theorem:application_sanov_type}
For all $(\psi_n)_{n\in\N}$ with $\psi_n \in \cP_{emp}^n(\cS)$ and $\psi_n \rightarrow \psi$ 
 the sequence \\
 $(\eta_n(\psi_n,\cdot))_{n\in\N}$ satisfies the large deviation principle with rate function $I: \cP(\cR) \rightarrow [0,\infty]$, given by 
\begin{align}
\label{eqn:rate_function_sanov_type}
I(\phi) = 
\inf_{ \xi \in \margi^{-1} (\{(\phi,\psi)\}) } H( \xi |\lambda)
- \inf_{ \xi \in \margi^{-1} (\cP(\cR) \times \{\psi\}) } H( \xi |\lambda).
\end{align}
$I$ is continuous on $[I<\infty]$. 
\end{theorem}

As $\cP(\cS)$ is first countable, it is sufficient to show that 
\ref{item:uniform_condition_lower_bound} and
\ref{item:uniform_condition_upper_bound} 
of Theorem \ref{theorem:equivalent_notions_bounds_prodRCP} hold. 
In \ref{obs:sufficient_bound_for_ldp_finite_sanov_type} we use the bounds of Lemma \ref{lemma:probabilits_single_empirical_estimates} to derive other bounds which imply \ref{item:uniform_condition_lower_bound} and
\ref{item:uniform_condition_upper_bound}. 
The continuity of $I$ follows by continuity of the map $\nu \mapsto H(\nu |\lambda)$ (Lemma \ref{lemma:continuits_H_on_support_H}). 

\begin{lemma} \cite[Lemma 2.1.9]{DeZe10}
\label{lemma:probabilits_single_empirical_estimates}
For $\nu \in \cP_{emp}^n(\cR\times \cS)$ one has, with $M=\#\cR\#\cS$,
\begin{align}
(n+1)^{-M}  e^{-n H(\nu | \lambda)}
\le \left[ \bigotimes_{i=1}^n \lambda \right]( L_n^{-1}( \{\nu\}) )
\le e^{-n H(\nu | \lambda)}.
\end{align}
\end{lemma}

\begin{obs}
\label{obs:sufficient_bound_for_ldp_finite_sanov_type}
From Lemma \ref{lemma:probabilits_single_empirical_estimates} we obtain the following bounds for $A\in \cB(\cP(\cR))$ and $B\in \cB(\cP(\cS))$. 
\begin{align}
\mu_n(A\times B) 
\notag & \le \# L_n^{-1}(A) \# L_n^{-1}(B) 
e^{-n \inf_{\nu \in \margi^{-1} (A\times B) \cap \cP_{emp}^n(\cR\times \cS) } H(\nu |\lambda)} \\
\label{eqn:inequalits_upper_A_times_B}
& \le (n+1)^{M} e^{-n \inf_{\nu \in \margi^{-1} (A\times B) } H(\nu |\lambda)}, \\
\mu_n(A\times B) 
& \ge (n+1)^{-M}
e^{-n \inf_{\nu \in \margi^{-1} (A\times B) \cap \cP_{emp}^n(\cR\times \cS) } H(\nu |\lambda)}.
\label{eqn:inequalits_lower_A_times_B}
\end{align}
Whence 
\begin{align}
\notag & \tfrac1n \log \Big[ (n+1)^{-2M} \Big] - \left[\inf_{\nu \in \margi^{-1} (A\times B) \cap \cP_{emp}^n(\cR\times \cS) } H(\nu |\lambda) - \inf_{ \xi \in \margi^{-1} (\cR\times B) } H( \xi |\lambda) \right]\\
\notag & \le
\tfrac1n \log 
\mu_n(A\times \cS | \cR \times B) \\
& \le \tfrac1n \log \Big[ (n+1)^{2M} \Big]
- \left[\inf_{\nu \in \margi^{-1} (A\times B) } H(\nu |\lambda) - \inf_{ \xi \in \margi^{-1} (\cR\times B) \cap \cP_{emp}^n(\cR\times \cS) } H( \xi |\lambda)\right].
\label{eqn:inequalities_finite_n_empiricals}
\end{align}
In order to derive \ref{item:uniform_condition_lower_bound} and
\ref{item:uniform_condition_upper_bound} 
of Theorem \ref{theorem:equivalent_notions_bounds_prodRCP} we make the following observation. 
By \eqref{eqn:inequalities_finite_n_empiricals} 
we have for an open $U$ and a closed $W$ that 
if  for both $A=U$ and $C= \cR$ as well as $A=\cR$ and $C=W$ we have 
\begin{align}
\notag &  \inf_{V_0\in \cN_\psi} \limsupn \sup_{ \putatop{ V\in \cH , V \subset V_0}{ V\cap \cP_{emp}^n(\cS) \ne \emptyset } }
 \left[\inf_{\nu \in \margi^{-1} (\overline A\times V) \cap \cP_{emp}^n(\cR\times \cS) } H(\nu |\lambda) - \inf_{ \xi \in \margi^{-1} (C^\circ \times V) } H( \xi |\lambda) \right]\\
 & \le \inf_{ \nu \in \margi^{-1} (A\times \{\psi\}) } H( \nu |\lambda) - \inf_{ \xi \in \margi^{-1} (C\times \{\psi\}) } H( \xi |\lambda),
 \label{eqn:desired_bound} 
\end{align}
then
\begin{align}
\sup_{V_0\in \cN_\psi } \liminfn \inf_{ \putatop{ V\in \cH , V \subset V_0}{ V\cap \supp(\mu_n \circ \pi^{-1}) \ne \emptyset } }
\tfrac1n \log \mu_n( \overline  U \times \cS | \cR \times V) 
 \ge - \inf I(U), \\
  \inf_{V_0\in \cN_\psi } \limsupn \sup_{ \putatop{ V\in \cH , V \subset V_0}{ V\cap \supp(\mu_n \circ \pi^{-1}) \ne \emptyset } }
\tfrac1n \log \mu_n( W^\circ \times \cS | \cR \times V) \le - \inf I(W),
\end{align}
As 
\begin{align}
\notag &  \inf_{V_0\in \cN_\psi} \limsupn \sup_{ \putatop{ V\in \cH , V \subset V_0}{ V\cap \cP_{emp}^n(\cS) \ne \emptyset } }
 \left[\inf_{\nu \in \margi^{-1} (A\times V) \cap \cP_{emp}^n(\cR\times \cS) } H(\nu |\lambda) - \inf_{ \xi \in \margi^{-1} (C\times V) } H( \xi |\lambda) \right] \\
\notag & \quad \le \inf_{V_0\in \cN_\psi} \limsupn \sup_{ \putatop{ V\in \cH , V \subset V_0}{ V\cap \cP_{emp}^n(\cS) \ne \emptyset } }
\inf_{\nu \in \margi^{-1} ( A \times V) \cap \cP_{emp}^n(\cR\times \cS) } H(\nu |\lambda) \\
\notag & \hspace{3cm} - 
 \sup_{V_0\in \cN_\psi}  \inf_{ \putatop{V\in \cH, V \subset V_0}{V\cap \supp(\mu_n \circ \pi^{-1})\ne \emptyset} }
\inf_{ \xi \in \margi^{-1} (C \times V) } H( \xi |\lambda) \\
\notag & \quad \le 
\inf_{V_0\in \cN_\psi} \limsupn \sup_{ \zeta \in \cP_{emp}^n(\cS) \cap V_0 \ }
\inf_{\nu \in \margi^{-1} (A\times \{\zeta\}) \cap \cP_{emp}^n(\cR\times \cS) } H(\nu |\lambda) 
\\
& \hspace{3cm} - 
 \sup_{V\in \cN_\psi}  \inf_{ \xi \in \margi^{-1} (C \times V) } H( \xi |\lambda), 
\end{align}
\eqref{eqn:desired_bound} holds (for both $A=U$ and $C= \cR$ as well as for $A=\cR$ and $C=W$, where $U$ is open and $W$ is closed)
if for all open $U$ and all closed $W$
\begin{align}
\label{eqn:inequalits_first_part}
\inf_{V_0\in \cN_\psi} \limsupn \sup_{ \zeta \in \cP_{emp}^n(\cS) \cap V_0 \ }
\notag 
& \inf_{\nu \in \margi^{-1} (U\times \{\zeta\}) \cap \cP_{emp}^n(\cR\times \cS) } H(\nu |\lambda) \\
& \le  \inf_{\nu \in \margi^{-1} ( U \times \{\psi\})  } H(\nu |\lambda), \\
\label{eqn:inequalits_second_part}
 \sup_{V\in \cN_\psi}  
\inf_{ \xi \in \margi^{-1} (W \times V) } H( \xi |\lambda)
 & \ge  \inf_{ \xi \in \margi^{-1} (W \times \{\psi\}) } H( \xi |\lambda). 
\end{align}
\eqref{eqn:inequalits_second_part} is a consequence of Lemma \ref{lemma:convergence_infimum_rate_function_on_schrinking_neighbourhoods}, as $\margi^{-1}(W\times V) = \margi^{-1}(W\times \cP(\cS)) \cap \margi^{-1}(\cP(\cR) \times V)$, the set $F= \margi^{-1}(W\times \cP(\cS))$ is closed for closed $W$, $\margi^{-1}(\cP(\cR) \times V)= (\pi \circ \margi)^{-1}(V)$ and $\pi \circ \margi$ is continuous. 
The proof of inequality \eqref{eqn:inequalits_first_part} requires a little more attention. 
First we present some facts which are used to prove this inequality in Lemma \ref{lemma:inequalits_first_part}. 
\end{obs}

\begin{lemma}
\label{lemma:continuits_H_on_support_H}
\cite[Remark below Definition 2.1.5]{DeZe10}
The map $\nu \mapsto H(\nu |\lambda)$ is continuous on $[H(\cdot|\lambda) <\infty]$. 
In particular, for all $\epsilon>0$ and $\xi \in \cP(\cR \times \cS)$ there exists a $\Theta \in \cN_\xi$ such that 
\begin{align}
H(\nu |\lambda) - \epsilon \le H( \xi | \lambda) \qquad (\nu \in \Theta \cap [H(\cdot|\lambda) <\infty]). 
\end{align}
Consequently, $I$ as in \eqref{eqn:rate_function_sanov_type} is continuous on $[J<\infty]$.  
\end{lemma}

\begin{lemma}
\label{lemma:large_enough_n_then_empricial_in_open_set}
\begin{enumerate}
\item 
\label{item:empirical_multiples}
Let $k,l\in\N$ and $\zeta \in \cP_{emp}^k(\cS)$. For all $m\ge kl$ there exists a $\nu \in \cP_{emp}^m(\cS)$ such that $\fd(\nu,\zeta) <\frac{1}{l}$.
\item 
\label{item:large_enough_denseness_empiricals}
For all open $\Theta \subset\cP(\cS)$ there exists an $N\in\N$ such that $\cP_{emp}^n(\cS) \cap \Theta \ne \emptyset$ for all $n\ge N$.
\end{enumerate}
\end{lemma}
\begin{proof}
\ref{item:empirical_multiples}
Let $i \in \{1,\dots,k\}$. 
Let  $\xi \in \cP_{emp}^i(\cS)$. 
Then the measure $\frac{lk}{lk+i}\zeta + \frac{i}{lk+i} \xi$ is an element of $\cP_{emp}^{lk+i}(\cS)$. 
For every $A \subset \cS$
\begin{align}
\left| [\tfrac{lk}{lk+i}\zeta + \tfrac{i}{lk+i} \xi](A) - \zeta(A) \right| 
\le 2 \tfrac{i}{lk+i} \le 2\tfrac{k}{lk} = \tfrac2l.
\end{align}
By definition of the Prohorov metric, this implies $\fd( [\tfrac{lk}{lk+i}\zeta + \tfrac{i}{lk+i} \xi], \zeta ) \le \frac{2}{l}$. 

\ref{item:large_enough_denseness_empiricals}
Let $\xi \in \cP(\cS)$ and $\delta>0$ be such that $B(\xi,\delta) \subset \Theta$. 
For each $\xi \in \cP(\cS)$ there is a $k\in\N$ and a $\zeta \in \cP_{emp}^k(\cS)$ such that $\fd(\zeta,\xi)<\frac{\delta}{2}$. Because of this \ref{item:large_enough_denseness_empiricals} follows from \ref{item:empirical_multiples} by letting $l$ be such that $\frac{1}{l} < \frac{\delta}{2}$ and $N= lk$. 
\end{proof}

\begin{lemma}
\label{lemma:existence_nu_with_properties}
Let $\xi \in \cP(\cR \times \cS)$, $\pi \circ \margi (\xi) = \psi$ and $\xi \ll \lambda$. 
For all $\delta>0$ 
there exists 
a $\kappa>0$ and 
an $N\in\N$ such that for all $n\ge N$ and all $\zeta \in \cP_{emp}^n(\cS)$ with $\fd(\zeta,\psi) < \kappa$ 
there is a $\nu \in \cP_{emp}^n(\cR \times \cS)$ with 
\begin{align}
\pi \circ \margi ( \nu ) = \zeta, \quad 
\nu \ll \lambda, \quad
\fd (\nu, \xi) < \delta, \quad 
\fd \Big(\nu(\cdot \times \cS) , \xi(\cdot \times \cS) \Big) <\delta. 
\end{align}
\end{lemma}
\begin{proof}
In this proof, for a measure $\xi \in \cP(\cR \times \cS)$, we write 
$\xi_{rs} = \xi(\{(r,s)\})$, so that $\xi = \sum_{rs} \xi_{rs} \delta_{(r,s)}$ where we use the short-hand notation ``$\sum_{rs}$'' instead of ``$\sum_{r\in\cR,s\in \cS}$''. 
Let $M= \# \cR \# \cS$. 
Note that 
\begin{align}
\label{eqn:prohorov_metric_estimate}
\fd(\xi,\nu) \le M \max_{r\in \cR, s\in \cS} |\xi_{rs} - \nu_{rs}| \qquad \big(\xi,\nu \in \cP(\cR \times \cS) \big) .
\end{align}

Let $\kappa>0$ and $n\in\N$. 
We first give an estimation by which it is clear which $\kappa$ and $N$ one should choose. 
By the assumptions on $\lambda$ for every $s\in \cS$ there exists a $r_s \in \cR$ with $\lambda_{r_s s}>0$. 

First we show that there exists a $\xi^* \in \cP_{emp}^n(\cX \times \cY)$ with $\xi^* \ll \xi$ and   $|\xi_{rs}^* - \xi_{rs}| \le \frac{2}{n}$ for all $r\in \cR$ and $s\in \cS$. 
For each pair $(r,s) \in \cR \times \cS$ with $\xi_{rs}>0$ we can choose a $\xi_{rs}'\in \{0,\frac1n,\frac2n, \dots, 1\}$ such that $|\xi_{rs} - \xi_{rs}'|<\frac1n$. 
By letting $\xi_{rs}^*=0$ when $\xi_{rs}=0$ and add or subtract $\frac1n$ to some of the $\xi_{rs}'$ we obtain a collection of $\xi_{rs}^*\in \{0,\frac1n,\frac2n,\dots,1\}$ with $\sum_{rs} \xi_{rs}^*=1$ and $|\xi_{rs}^* - \xi_{rs}| \le \frac{2}{n}$ and $\xi^*_{rs}=0$ whenever $\xi_{rs}=0$ for all $r\in \cR$ and $s\in \cS$. 

Let $\xi \in \cP(\cR \times \cS)$. 
Suppose that $\zeta \in \cP_{emp}^n(\cS)$ is such that $|\zeta_s - \sum_r \xi_{rs}|<\kappa$. 
Then $|\zeta_s - \sum_r \xi^*_{rs}| < \kappa + \frac2n M$. 
We construct a $\nu \in \cP_{emp}^n(\cR\times \cS)$ by defining the $\nu_{rs}$ by each $s$ separately. 
Let $s\in S$. 
If $\zeta_s - \sum_r \xi^*_{rs}<0$, then we choose $\nu_{rs} \le \xi^*_{rs}$ with $\nu_{rs} \in \{0,\frac1n,\dots,1\}$ in such way that $\sum_r \nu_{rs} = \zeta_s$ (note that $|\nu_{rs} - \xi^*_{rs}| \le |\zeta_s - \sum_r \xi^*_{rs}|$). 
While, if $\zeta_s - \sum_r \xi^*_{rs}\ge 0$, then we let $\nu_{rs} = \xi^*_{rs}$ for all $r\ne r_s$ and we let $\nu_{r_s s} = \xi_{r_s s}^* + \zeta_s - \sum_r \xi^*_{rs}$ (so that $\sum_r \nu_{rs} = \zeta_s$). 
As $\xi^* \ll \xi$ and $\xi \ll \lambda$, by the construction of $\nu$ we have $\nu \ll \lambda$. 
Moreover, we have 
$\pi \circ \margi(\nu) = \zeta$ and 
\begin{align}
 \max_{r\in \cR, s\in \cS} \big| \nu_{rs} -  \xi_{rs} \big| 
\notag & \le  \max_{s\in \cS} \Big|\zeta_s - \sum_r \xi^*_{rs}\Big| 
+
 \max_{r\in \cR, s\in \cS} |\xi^*_{rs} - \xi_{rs}|  \\
& \le \kappa + \tfrac2n M+ \tfrac2n . 
\end{align}
Which implies by \eqref{eqn:prohorov_metric_estimate}
\begin{align}
\fd(\nu,\xi) 
 \le M \kappa + \tfrac{2}{n} (M^2+M). 
\end{align}
Moreover, as $| \sum_s \nu_{rs} - \sum_s \xi_{rs}|  \le M \max_{s\in \cS} |\nu_{rs} -  \xi_{rs}|$, 
\begin{align}
\fd \Big(\nu(\cdot \times \cS) , \xi(\cdot \times \cS) \Big) 
\le M\max_{r\in \cR} \Big| \sum_s \nu_{rs} - \sum_s \xi_{rs} \Big| 
\le M^2 \kappa + \tfrac{2}{n} (M^3 +M^2). 
\end{align}
By choosing $\kappa>0$ and $N\in\N$ such that $M^2 \kappa + \tfrac{2}{n} (M^3 +M^2) <\delta$ the proof is complete. 
\end{proof}

\begin{lemma}
\label{lemma:inequalits_first_part}
For all open $U\subset \cR$ 
\begin{align}
0 \le \inf_{V_0\in \cN_\psi} \limsupn \sup_{ \zeta \in \cP_{emp}^n(\cS) \cap V_0 \ }
\notag & \inf_{\nu \in \margi^{-1} (U\times \{\zeta\}) \cap \cP_{emp}^n(\cR\times \cS) } H(\nu |\lambda) \\
& \le  \inf_{\nu \in \margi^{-1} ( U \times \{\psi\})  } H(\nu |\lambda).
\end{align}
\end{lemma}
\begin{proof}
We assume $\inf_{\nu \in \margi^{-1} ( U \times \{\psi\})  } H(\nu |\lambda)<\infty$. 
Let $\xi \in \margi^{-1}(U \times \{\psi\})$ be such that $H(\xi |\lambda)<\infty$. 
Let $\epsilon>0$. 
We show there exists a $V_0 \in \cN_\psi$ and an $N\in\N$ such that for all $n\ge N$
the set $\cP_{emp}^n(\cS)\cap V_0$ is not empty and for all $\zeta \in \cP_{emp}^n(\cS)\cap V_0$ 
 there exists a $\nu \in \margi^{-1}(U\times \{\zeta\})\cap \cP_{emp}^n(\cR \times \cS)$ with 
\begin{align}
\label{eqn:H_nu_minus_epsilon_le_H_xi_2}
H(\nu | \lambda) - \epsilon \le H(\xi | \lambda). 
\end{align}
Let $\delta$ be such that (see Lemma \ref{lemma:continuits_H_on_support_H})
\begin{align}
\label{eqn:ball_included_in_U}
& B( \xi(\cdot \times \cS),\delta) \subset U, \\
\label{eqn:upper_semicontinuits_on_domain_H}
& H(\nu | \lambda) - \epsilon \le H(\xi | \lambda) \qquad (\nu \in B(\xi,\delta)\cap [H(\cdot|\lambda) <\infty]). 
\end{align}
Then let $\kappa>0$ and $N\in\N$ be as in Lemma \ref{lemma:existence_nu_with_properties}. 
Let $V_0 = B(\psi,\kappa)$. 
By Lemma \ref{lemma:large_enough_n_then_empricial_in_open_set} we may assume that $N$ is large enough such that $\cP_{emp}^n (\cS) \cap V_0 \ne \emptyset$. 
Let $n\ge N$ and $\zeta \in \cP_{emp}^n (\cS) \cap V_0$. 
By Lemma \ref{lemma:existence_nu_with_properties} there exists a $\nu \in \cP_{emp}^n(\cR \times \cS)$ with $\pi \circ \margi(\nu) = \zeta$, $\nu \ll \lambda$ and 
$\nu(\cdot \times \cS) \in B( \xi(\cdot \times \cS), \delta)$, 
$\nu \in B(\xi,\delta)$, 
i.e., by \eqref{eqn:ball_included_in_U}, $\nu \in \margi^{-1}(U\times \{\zeta\})$. 
$\nu \ll \lambda$ implies $\nu \in [H(\cdot|\lambda)<\infty]$, thus with \eqref{eqn:upper_semicontinuits_on_domain_H} we obtain \eqref{eqn:H_nu_minus_epsilon_le_H_xi_2}. 
\end{proof}

\section{Examples}
\label{section:examples}


In Section \ref{section:application_sanov_type} we showed that the regular conditional kernel $\eta_n$ as in \eqref{eqn:eta_n_for_sanov_type} satisfies 
\ref{item:forall_y_seq_cond_lower_bound}
and 
\ref{item:forall_y_seq_cond_upper_bound} 
of Theorem \ref{theorem:equivalent_notions_bounds_prodRCP}
by showing that 
\ref{item:uniform_condition_lower_bound} and
\ref{item:uniform_condition_upper_bound} of that theorem hold. 
This is not always the most optimal approach; in Example \ref{example:normal_distribution_conditioned} 
we show that for a specific example of Gaussian measures 
the expression of $\eta_n$
allows us to derive 
\ref{item:forall_y_seq_cond_lower_bound}
and 
\ref{item:forall_y_seq_cond_upper_bound} 
directly. 

Furthermore, relying on Theorem \ref{theorem:convex_combi_with_exp_distr_of_two_sequences}, in Example \ref{example:failure_1},
we give an example of a $(\eta_n)_{n\in\N}$ for which 
\ref{item:forall_y_seq_cond_lower_bound}
of Theorem \ref{theorem:equivalent_notions_bounds_prodRCP}
does not hold. 
In Remark \ref{remark:quenched_fail} 
we mention that for the one choice of measures in Example \ref{example:failure_1} 
a quenched large deviation principle is satisfied, while for the other choice of measures 
there is no quenched large deviation principle.
In Example \ref{example:failure_2} we show that for a choice of measures as in Example \ref{example:failure_1} the conditional regular kernel in a specific chosen point does not  satisfy any large deviation principle. 
In Remark \ref{remark:exponential_tightness} we discuss exponential tightness of the regular conditional kernel. 
In Remark \ref{remark:on_CoSc} we discuss the differences between the present paper and the paper of La Cour and Schieve \cite{CoSc15}.

\begin{example}
\label{example:normal_distribution_conditioned}
Let $r\ne 0$, $Z_n : = \int_{\R} \int_\R e^{-\frac{n}{2}  (x^2 - 2rxy +y^2)} \D x \D y$ and 
consider $(\mu_n)_{n\in\N}$ the sequence of probability measures on $\cB(\R \times \R)$ determined by 
\begin{align}
\mu_n(A \times B) = \frac{1}{Z_n} \int_{\R} \int_\R \1_{A\times B}(x,y) e^{-\frac{n}{2} (x^2 - 2rxy +y^2)} \D x \D y \qquad (A,B \in \cB(\R)). 
\end{align}
The sequence satisfies the large deviation principle with rate function $J: \R^2 \rightarrow [0,\infty]$ given by 
$
J(x,y) = \tfrac12 (x^2 - 2rxy +y^2)
$. 
By Theorem \ref{theorem:mixing_with_function_gives_weakly_continuous_kernel} $\eta_n$  given by 
\begin{align}
\eta_n(y,A) = \frac{\int_A e^{- \frac{n}{2} (x^2 - 2rxy)} \D x}{\int_\R e^{- \frac{n}{2} (x^2 - 2rxy)} \D x}
=  \frac{\int_A e^{- \frac{n}{2} (x- ry)^2} \D x}{\int_\R e^{- \frac{n}{2} (x- ry)^2} \D x},
\end{align}
is the weakly continuous product regular conditional probability under $\mu_n$ with respect to the projection on the $\cY$-coordinate. 
If $y_n \rightarrow y$, one can show that for $\lambda \in \R$ 
\begin{align}
\limn \tfrac1n \log \int_\R e^{n\lambda x} \DD [\eta_n(y_n,\cdot)](x)
\notag & =\limn \tfrac1n \log  \int_\R e^{n\lambda x} \DD [\eta_n(y,\cdot)](x) \\
& = \lambda r y + \tfrac12 \lambda^2.
\end{align}
Then by the G\"{a}rtner-Ellis Theorem (see for example Dembo and Zeitouni \cite[Theorem 2.3.6]{DeZe10}) we conclude that $(\eta_n(y_n,\cdot))_{n\in\N}$ satisfies the large deviation principle with the same rate function as the one of the large deviation principle of $(\eta_n(y,\cdot))_{n\in\N}$, which is $x\mapsto (x - ry )^2$.
Note that this equals $J(x,y) - \inf J(\R \times \{y\})$ because of the equality $x^2 - 2rxy +y^2= (x-ry)^2 + ( 1-r^2)y^2$. 
\end{example}

The proof of the following theorem can be found in Appendix \ref{section:proof_theorem_examples}. 

\begin{theorem}
\label{theorem:convex_combi_with_exp_distr_of_two_sequences}
Let $\cX$ and $\cY$ be separable metric spaces. 
Let $(\mu_n^1)_{n\in\N}$ and $(\mu_n^2)_{n\in\N}$ be sequences of probability measures on $\cB(\cX)$. 
Let $(\nu_n)_{n\in\N}$ be a sequence of probability measures on $\cB(\cY)$ that satisfies the large deviation principle with a rate function $L: \cY \rightarrow [0,\infty]$.
Suppose that $y\in \cY$ and $W_n\in \cN_y$ are such that $\bigcap_{n\in\N} W_n = \{y\}$ and 
$\alpha_n : \cY \rightarrow [0,1]$ is a continuous function with $\alpha_n(y) =0$ and $\alpha_n =1$ on $\cY \setminus W_n$ such that 
\begin{align}
\label{eqn:liminf_lambda_n_integrals_equal_zero}
\liminfn \tfrac1n \log ( \int_{W_n} \alpha_n \D \nu_n )=0, \qquad 
\liminfn \tfrac1n \log ( \int_{W_n} (1- \alpha_n) \D \nu_n)=0. 
\end{align}
Assume $(\mu_n^1)_{n\in\N}$ satisfies the large deviation principle with rate function $I$. 
Assume furthermore that for all open $A\subset \cX$
\begin{align}
\label{eqn:logaritmic_inequality_mu_1_and_mu_2_liminf}
\liminfn \tfrac1n \log \mu_n^1(A) &\ge \liminfn \tfrac1n \log \mu_n^2(A) , \\
\limsupn \tfrac1n \log \mu_n^1(\cX \setminus A) &\ge \limsupn \tfrac1n \log \mu_n^2(\cX \setminus A).
\label{eqn:logaritmic_inequality_mu_1_and_mu_2_limsup}
\end{align}
Let $\mu_n$  be the probability measure on $\cB(\cX) \otimes \cB(\cY)$ for which for $A\in \cB(\cX)$, $B\in \cB(\cY)$ 
\begin{align}
\mu_n(A\times B) = \mu_n^1(A) \int_\cY \1_B \alpha_n  \D \nu_n
+ \mu_n^2(A) \int_\cY \1_B (1- \alpha_n)  \D \nu_n .
\end{align}
Then $(\mu_n)_{n\in\N}$ satisfies the large deviation principle with rate function $J: \cX \times \cY \rightarrow [0,\infty]$ given by $J(x,y) = I(x) + L(y)$. 
$\eta_n: \cY \times \cB(\cX) \rightarrow [0,1]$ defined by
\begin{align}
\eta_n(y,A) = \alpha_n(y) \mu_n^1(A) + (1- \alpha_n(y) ) \mu_n^2(A)
\end{align}
is the weakly continuous product regular conditional probability under $\mu_n$ with respect to $\pi : \cX \times \cY \rightarrow \cY$ given by $\pi(x,y)=y$.
\end{theorem}

 Note that $I(x) = J(x,y) - \inf J(\cX \times \{y\})$ for all $x\in \cX, y\in \cY$.

\begin{examples}
\label{examples:examples_of_nu_n_satisfying_conditions}
We give examples of $\cY, W_n, \alpha_n, \nu_n$ and $L$ such that \eqref{eqn:liminf_lambda_n_integrals_equal_zero} of Theorem \ref{theorem:convex_combi_with_exp_distr_of_two_sequences} is satisfied and $(\nu_n)_{n\in\N}$ satisfies the large deviation principle with rate function $L$. 
\begin{enumerate}
\item 
\label{item:nu_n_exponential_and_convex_with_dirac}
Let $\cY = [0,\infty)$, $\alpha_n (y) = \min\{ny,1\}$ for $y\in \cY$ and let $\nu_n(B) = \int_0^\infty \1_B(y) n e^{-ny} \D y$ for $B\in \cB([0,\infty))$. 
Then $\int_0^\frac1n \alpha_n  \D \nu_n = 1-2e^{-1}$ and $\int_0^\frac1n (1-\alpha_n) \D \nu_n = e^{-1}$. 
Therefore with this $\nu_n$, $\alpha_n$ and $W_n = [-\frac1n,\frac1n]$ \eqref{eqn:liminf_lambda_n_integrals_equal_zero} is satisfied. 
Moreover $(\nu_n)_{n\in\N}$  satisfies the large deviation principle with rate function $L :\cY \rightarrow [0,\infty]$,  $L(y) =y$ (this follows from example by the G\"{a}rtner-Ellis Theorem \cite[Theorem 2.3.6]{DeZe10}). 
\item 
\label{item:nu_n_normal_and_convex_with_dirac}
Let $\cY = \R$ and  $\nu_n =\mu_{\cN(0,\frac1n)}$ (the Gaussian measure corresponding to a $\cN(0,\frac1n)$ distributed random variable). 
Then there exists a decreasing sequence $(\epsilon_n)_{n\in\N}$ in $(0,\infty)$ with $\epsilon_n \downarrow 0$, such that with $W_n = [-\epsilon_n,\epsilon_n]$ there exist functions $\alpha_n$ as in Theorem \ref{theorem:convex_combi_with_exp_distr_of_two_sequences} such that 
\eqref{eqn:liminf_lambda_n_integrals_equal_zero} is satisfied (see the postscript). 
With $\nu_n^0 = \frac12 \delta_0 + \frac12 \nu_n$ instead of $\nu_n$, \eqref{eqn:liminf_lambda_n_integrals_equal_zero} is also satisfied. 
Moreover, $(\nu_n)_{n\in\N}$ and  $(\nu_n^0)_{n\in\N}$ (use Lemma \ref{lemma:finite_sum_in_log_becomes_max_in_limsup}) satisfy the large deviation principle with rate function $L :\cY \rightarrow [0,\infty]$,  $L(y) = \frac12 y^2$. \\
\underline{Postscript}.
Let $\beta =  \nu_1([-1,1])$. 
Let $\kappa_n = \frac{1}{\sqrt n}$. 
Then $\nu_n[-\kappa_n,\kappa_n] = \beta$ for all $n\in\N$. 
Let $\phi_\epsilon: \R \rightarrow [0,1]$ be defined by 
$\phi_\epsilon (z) = \min\{\epsilon^{-1} |z|,1\}$. 
Then 
$\lim_{\epsilon \downarrow 0} \int_{[-\kappa_n,\kappa_n]} \phi_\epsilon \D \nu_1 = \beta$, 
$\lim_{\epsilon \downarrow 0} \int_{[-\kappa_n,\kappa_n]} 1- \phi_\epsilon \D \nu_1 =0$
 and 
\begin{align}
\int_{[-\kappa_n,\kappa_n]} \phi_{\kappa_n} \D \nu_n < \int_{[-\kappa_n,\kappa_n]} 1- \phi_{\kappa_n} \D \nu_n.
\end{align}
Therefore, for all $n\in\N$, there exists an $\epsilon_n \in (0,\kappa_n)$ such that 
\begin{align}
\int_{[-\kappa_n,\kappa_n]} \phi_{\epsilon_n} \D \nu_n = \tfrac12 \beta = \int_{[-\kappa_n,\kappa_n]} 1- \phi_{\epsilon_n} \D \nu_n,
\end{align}


With $\alpha_n = \phi_{\epsilon_n}$, \eqref{eqn:liminf_lambda_n_integrals_equal_zero} as in Theorem \ref{theorem:convex_combi_with_exp_distr_of_two_sequences} is satisfied. 
\end{enumerate}
\end{examples}

\begin{example}
\label{example:failure_1}
With $\cX= \R$, $\mu_n^1 = \mu_{\cN(0,\frac1n)}$,
$\mu_n^2 = \delta_{\frac1n}$ and $I(x) = \frac12 x^2$ for $x\in \R$ and 
$\cY,\nu_n$ (or $\nu_n^0$), $\alpha_n$, $W_n$ and $L$ as in
Examples \ref{examples:examples_of_nu_n_satisfying_conditions}
\ref{item:nu_n_exponential_and_convex_with_dirac} or \ref{item:nu_n_normal_and_convex_with_dirac}
the conditions of Theorem \ref{theorem:convex_combi_with_exp_distr_of_two_sequences} are satisfied (note that $(\delta_{\frac1n})_{n\in\N}$ satisfies the large deviation principle with rate function $H: \R \rightarrow [0,\infty]$ given by $H(0)=0$ and $H(x) =\infty$ for $x\ne 0$). 

Then $\eta_n(0,\cdot) = \delta_{\frac1n}$ and $\eta_n(\epsilon_n,\cdot) = \mu_{\cN(0,\frac1n)}$ for all $n\in\N$. Whence $(\eta_n(0,\cdot))_{n\in\N}$ satisfies the large deviation principle with rate function $H$ and 
$(\eta_n(\epsilon_n,\cdot))_{n\in\N}$ (and also $(\eta_n(y,\cdot))_{n\in\N}$ for $y>0$) satisfies the large deviation principle with rate function $I$. 
Because $I \ge H$, the sequence $(\eta_n(0,\cdot))_{n\in\N}$ satisfies the large deviation upper bound not only with $I$ but also with $H$ instead of $I$. 
Therefore {\ref{item:ldp_upper_bound}} of Theorem \ref{theorem:lower_and_upper_bound_eta_n_convergent_y_n} holds in case $y_n =0$ for all $n$. 
Since $(\eta_n(0,\cdot))_{n\in\N}$ does not satisfy the large deviation principle with rate function $I$, 
 {\ref{item:ldp_lower_bound}} of 
Theorem \ref{theorem:lower_and_upper_bound_eta_n_convergent_y_n}
does not hold. Therefore for any decreasing sequence $(V_m)_{m\in\N}$
in $\cN_0$ with $\bigcap_{m\in\N} V_m= \{0\}$ there exists an open set $U$ with $\inf I(U) <\infty$ with 
\begin{align}
\liminfn \limsupm \tfrac1n \log \mu_n( U \times \cY | \cX \times V_m) < -\inf I(U). 
\end{align}
We illustrate this for $\cY,\alpha_n, W_n, \nu_n $ and $L$ as in Examples \ref{examples:examples_of_nu_n_satisfying_conditions}\ref{item:nu_n_exponential_and_convex_with_dirac}:
For $V_m=[0,\tfrac1m)$, $U=(1,\infty)$ we get for $m\ge n$
\begin{align}
\mu_n(U \times V_m)  & = \mu_{\cN(0,\frac1n)}(U) \int_0^\frac1m n y\cdot n e^{-ny} \D y, \\
\mu_n(\cX \times V_m) & = \int_0^\frac1m n e^{-ny} \D y.
\end{align}
Since $\int_0^\frac1m n y\cdot n e^{-ny} \D y \le \frac{n}{m} \int_0^\frac1m  n e^{-ny} \D y$ we get
\begin{align}
\mu_n(U\times \cY |\cX \times V_m) \le \tfrac{n}{m} \mu_{\cN(0,\frac1n)}(U)
\end{align}
which converges to zero as $m\rightarrow \infty$, which implies
\begin{align}
\limsupm  \tfrac1n \log \mu_n(U\times \cY |\cX \times V_m) = -\infty < - \tfrac12 = - \inf I(U). 
\end{align}
\end{example}

\begin{remark}[Quenched large deviations]
\label{remark:quenched_fail}
Consider the situation as in Example \ref{example:failure_1}. 
For all $n\in\N$ we have the following. If
 $\zeta_n: \cY \times \cB(\cX) \rightarrow [0,1]$ is a product regular conditional probability under $\mu_n$ with respect to $\pi$, then $\zeta_n(y,\cdot) = \eta_n(y,\cdot)$ for $[\mu_n \circ \pi^{-1}]$-almost all $y$ (see Remark \ref{remark:almost_everywhere_uniqueness}).

Whence, with $\nu_n$ as in Examples \ref{examples:examples_of_nu_n_satisfying_conditions} \ref{item:nu_n_exponential_and_convex_with_dirac} or \ref{item:nu_n_normal_and_convex_with_dirac}, we have a quenched large deviation principle of the conditional probability with respect to the second coordinate with rate function $I$; for every product regular conditional probability 
$\zeta_n$ under $\mu_n$ with respect to $\pi$ there exists a $Z\subset \cY$ with $\mu_n \circ \pi^{-1}(Z) = \nu_n(Z)=1$ such that $(\zeta_n(y,\cdot))_{n\in\N}$ satisfies the large deviation principle with rate function $I$ for all $y\in Z$. 

However, with $\nu_n^0$ 
as in Examples \ref{examples:examples_of_nu_n_satisfying_conditions}\ref{item:nu_n_normal_and_convex_with_dirac} instead of $\nu_n$ for such $\zeta$ one has $\zeta_n(0,\cdot) = \eta_n(0,\cdot)$ as $\nu_n^0(\{0\}) >0$. Thus in this case we do not have such a quenched large deviation principle. 
\end{remark}

\begin{example}
\label{example:failure_2}
With $\cX = \N$, $\mu_n^1 = \sum_{k\in\N} 2^{-k} \delta_k$, $\mu_n^2= \delta_n$ and $I(x) =0$ for $x\in \N$ as in  Example \ref{example:failure_1}, and $\cY, W_n, \alpha_n, \nu_n$ and $L$ as in
Examples \ref{examples:examples_of_nu_n_satisfying_conditions}\ref{item:nu_n_exponential_and_convex_with_dirac} or \ref{item:nu_n_normal_and_convex_with_dirac}, the conditions of Theorem \ref{theorem:convex_combi_with_exp_distr_of_two_sequences} are satisfied. 
In this case $(\eta_n(0,\cdot))_{n\in\N}$ does not satisfy a large deviation principle. 
\end{example}

\begin{remark}
\label{remark:exponential_tightness}
\textbf{(Exponential tightness of the regular conditional kernel).} \\
Considering the situation as in Theorem \ref{theorem:convex_combi_with_exp_distr_of_two_sequences}, we would like to mention that if $(\mu_n^1)_{n\in\N}$ is exponentially tight, then so is $(\mu_n)_{n\in\N}$ since $\mu_n(K_1^c\times K_2^c) = \mu_n^1(K_1^c) \nu_n(K_2^c)$ for large $n$ and (compact) $K_1\subset \cX, K_2\subset \cY$. Similarly $(\eta_n(y,\cdot))_{n\in\N}$ is exponentially tight for all $y>0$ since $\eta_n(y,K^c) = \mu_n^1(K^c)$ for large $n$ and compact $K\subset \cX$. 
However, as is the case in Example \ref{example:failure_2}, $(\eta_n(y_n,\dot))_{n\in\N}$ need not be exponentially tight for all converging sequences $(y_n)_{n\in\N}$ (e.g., if $(\mu_n^2)_{n\in\N}$ is not exponentially tight, then $(\eta_n(0,\cdot))_{n\in\N}$ is neither). 
\end{remark}

\begin{remark}
\label{remark:on_CoSc}
Example \ref{example:failure_1} with $\nu_n$ (or $\nu_n^0$) and $\alpha_n$ as in Examples \ref{examples:examples_of_nu_n_satisfying_conditions}\ref{item:nu_n_normal_and_convex_with_dirac} fits the assumptions
 made in Section 4 of La Cour and Schieve \cite{CoSc15}.\footnote{The logarithmic moment generating function (see Dembo and Zeitouni \cite[Assumption 2.3.2]{DeZe10}) is given by $(x,y) \mapsto \frac12 x^2 + \frac12 y^2$, whence the Hessian of it equals the identity matrix and is therefore invertible. 
In \cite{CoSc15} is mentioned that one can not proceed the conditioning on all elements, 
but only those that equal the derivative of $y\mapsto \frac12 y^2$ at a certain point are considered, of which $0$ is an example.} 
In that paper it is claimed 
that the law of the first coordinate conditioned on the second coordinate
satisfies the large deviation principle with the rate function $I$.
Their notion of conditioning on $y$ is ``condition on an arbitrarily small neighbourhood around $y$''.
This approach needs to be justified. Our results are different, as by Example \ref{example:failure_1} the conditioned kernel in $0$, $\eta_n(0,\cdot)$ does not satisfy the large deviation principle with the rate function $I$ (even in the sense of quenched large deviations as discussed in Remark \ref{remark:quenched_fail}). 
\end{remark}

\appendix


\section{An elementary fact about limsup and liminf}
\label{appendix:elemtary_fact}

\begin{lemma}
\label{lemma:finite_sum_in_log_becomes_max_in_limsup}
Let $k\in\N$ and $a_n^i\in [0,\infty)$ for all $n\in\N$ and $i\in \{1,\dots,k\}$. 
If there exists an $N\in\N$ such that $\max_{i\in \{1,\dots,k\}} a_n^i >0$ for all $n\ge N$, then \footnote{Equation \eqref{eqn:limsup_sum_and_max} can also be found in Dembo and Zeitouni \cite[Theorem 1.2.15]{DeZe10}}
\begin{align}
\label{eqn:limit_difference_sum_and_max}
   &  \limn  \left( \traten \log \bigg(\sum_{i=1}^k a_n^i \bigg) - \max_{i\in \{1,\dots,k\}} \traten \log a_n^i\right)=0,  \\
\label{eqn:limsup_sum_and_max}
  &  \limsupn\traten \log \bigg(\sum_{i=1}^k a_n^i \bigg) = \max_{i\in \{1,\dots,k\}} \limsupn\traten \log a_n^i, \\
\label{eqn:liminf_sum_and_max}
  &  \liminfn \traten \log\bigg(\sum_{i=1}^k a_n^i \bigg) = \max_{i\in \{1,\dots,k\}} \liminfn \traten \log a_n^i.
\end{align}
\end{lemma}
\begin{proof}
\eqref{eqn:limit_difference_sum_and_max},\eqref{eqn:limsup_sum_and_max} and \eqref{eqn:liminf_sum_and_max} follow from the inequality 
\begin{align}
\label{eqn:max_le_sum_le_k_times_max}
\max_{i\in \{1,\dots,k\}} \tfrac1n \log a_n^i 
	\le \tfrac1n \log \bigg( \sum_{i=1}^k a_n^i  \bigg)
\notag &	\le  \tfrac1n \log ( k \max_{i\in \{1,\dots,k\}} a_n^i) \\
 & 	\le \tfrac1n \log k +  \max_{i\in \{1,\dots,k\}} \tfrac1n \log ( a_n^i).
\end{align}
\
\end{proof}

\section{Sufficient bounds for large deviation bounds}
\label{appendix:sufficient_bounds}

\assump{Let $\cX$ be a topological space. 
Let $I : \cX \rightarrow [0,\infty]$ have compact sublevel sets. 
Let $(\mu_n)_{n\in\N}$ be a sequence of probability measures on $\cB(\cX)$. 
}

\begin{lemma}
\label{lemma:supn_of_inf_decreasing_closed_sets}
Let $(F_m)_{m\in\N}$ be a decreasing sequence of closed sets with $F= \bigcap_{m\in\N} F_m$. 
Then 
\begin{align}
\supm \inf I(F_m) = \inf I(F). 
\end{align}
\end{lemma}
\begin{proof}
Let $c:= \supm \inf I(F_m) $. Note that $c\le \inf I(F)$. 
If $c=\infty$ there is nothing to prove. 
Assume that $c <\infty$. Let $K$ be the compact set $[I\le c]$. 
Then $F_m \cap K \ne \emptyset $ for all $m\in\N$, whence $F\cap K \ne \emptyset$ and thus $\inf I(F) \le c$. 
\end{proof}

\begin{obs}
For Lemma \ref{lemma:supn_of_inf_decreasing_closed_sets} the condition that $I$ has compact sublevel sets is not redundant. 
For example: 
Let $I: \N\cup \{0\} \rightarrow [0,\infty]$ be given by $I(0)=1$ and $I(x) = 0$ for $x\in \N$.
Then for $F_m = \{0\} \cup \{m,m+1,\dots\}$ and $F=\{0\}$ one has $\supm \inf I(F_m)=0$ and $\inf I(F)=1$. 
\end{obs}

\begin{lemma} 
\label{lemma:lower_bound_of_bigcup_upper_of_bigcap}\
\begin{enumerate}
\item $\cU$ be a set of open subsets of $\cX$. Suppose that for all $G\in \cU$
\begin{align}
\label{eqn:ldp_lower_bound}
 \liminfn \traten \log \mu_n(G) \ge -\inf  I(G) .
\end{align}
Then $G= \bigcup \cU$ satisfies \eqref{eqn:ldp_lower_bound} as well.\footnote{This can also be found in O'Brien \cite[Proposition 2.1]{OBr96}.}
\item Let $F_1,F_2,\dots$ be closed. Suppose that for all $F\in \{F_m:m\in\N\}$
\begin{align}
\label{eqn:ldp_upper_bound}
 \limsupn \traten \log \mu_n(F) \le -\inf I(F) .
\end{align}
Then $F= \bigcap_{m\in\N} F_m$ satisfies \eqref{eqn:ldp_upper_bound} as well. 
\end{enumerate}
\end{lemma}
\begin{proof}
\begin{align}
 \liminfn \traten \log 
 \mu_n \big(\bigcup \cU \big) 
 \ge \sup_{G\in \cU}  \liminfn \traten \log 
 \mu_n(G) \ge
 \sup_{G\in \cU} (-\inf  I(G)) , \\
 \limsupn \traten \log 
\mu_n \Big(\bigcap_{m\in\N} F_m \Big) 
 \le \inf_{m\in\N}  \limsupn  \traten \log \mu_n(F_m) 
 \le \inf_{m\in\N}  (-\inf  I(F_m)) . 
\end{align}
Now apply Lemma \ref{lemma:supn_of_inf_decreasing_closed_sets}.
\end{proof}


As a consequence of Lemma \ref{lemma:lower_bound_of_bigcup_upper_of_bigcap} we obtain the following.

\begin{theorem} 
\label{theorem:reducing_ldp_bounds_to_smaller_set}
Suppose that $\cG$ is a basis for the topology on $\cX$, such that  \eqref{eqn:ldp_lower_bound} holds for all $G\in \cG$ and \eqref{eqn:ldp_upper_bound} holds for $F= \cX \setminus G$. 
Suppose that every open $G$ can be written as countable union of elements in $\cG$. 
Then $(\mu_n)_{n\in\N}$ satisfies the large deviation principle with rate function $I$. 
\end{theorem}

\section{Proof of Theorem \ref{theorem:convex_combi_with_exp_distr_of_two_sequences}}
\label{section:proof_theorem_examples}

\begin{proof}[Proof of Theorem \ref{theorem:convex_combi_with_exp_distr_of_two_sequences}]

As $\cX$ and $\cY$ are separable metric spaces, every open subset of $\cX \times \cY$ is a countable union of elements of the form $A\times B$ where $A\subset \cX$ is open and $B\in \cH$, where (with $d_\cY$ the metric on $\cY$)
\begin{align}
 \cH 
 = \{  B(y,\delta) : \delta>0\} \cup 
\{ B(z,\delta) :  z\ne y, 0< \delta < d_{\cY}(y,z)\}. 
\end{align}
We use Theorem \ref{theorem:reducing_ldp_bounds_to_smaller_set} to prove the large deviation bounds. 
Note first that $( \cX \times \cY) \setminus (A \times B) = (\cX \times (\cY\setminus B))\cup ((\cX \setminus A) \times \cY)$, 
that $ \min\{\inf I(\cX \setminus A), \inf L(\cY \setminus B) \}
=  \inf_{(x,y) \in (\cX \times \cY) \setminus (A\times B)} I(x) + L(y)$ and that by \eqref{eqn:limsup_sum_and_max}
\begin{align}
\notag& \limsupn \tfrac1n  \log \mu_n(( \cX \times \cY) \setminus (A \times B)) \\
\notag  & 
\le \max \Big\{ \limsupn \tfrac1n  \log \mu_n( \cX \times (\cY\setminus B)),
\limsupn \tfrac1n  \log \mu_n( (\cX \setminus A) \times \cY ) \Big\}.  
\end{align}
Using this and Theorem \ref{theorem:reducing_ldp_bounds_to_smaller_set}
it is sufficient to show that for all open sets $A\subset \cX$ and $B\subset \cY$ 
\begin{align}
\limsupn \tfrac1n \log \mu_n( \cX \times (\cY \setminus  B)) 
\label{eqn:upper_bound_complement_B}
& \le - \inf L(\cY \setminus B), \\
\label{eqn:upper_bound_complement_A}
\limsupn \tfrac1n \log \mu_n( (\cX \setminus A) \times \cY ) 
 &\le - \inf I(\cX \setminus A)  \\
\label{eqn:lower_bound_A_times_B}
\liminfn \tfrac1n \log \mu_n(A\times B) 
& \ge  -\inf I(A) - \inf L(B). 
\end{align}
Let $A\subset \cX$ be open and $B\in \cH$. \\
$\bullet$ \eqref{eqn:upper_bound_complement_B} follows from the fact that $\mu_n(\cX \times (\cY \setminus B) ) = \nu_n ( \cY \setminus B)$. \\
$\bullet$  \eqref{eqn:upper_bound_complement_A} follows from the fact that by
\eqref{eqn:liminf_lambda_n_integrals_equal_zero},
\eqref{eqn:logaritmic_inequality_mu_1_and_mu_2_limsup} and \eqref{eqn:limsup_sum_and_max} we have
\begin{align}
\notag \limsupn \tfrac1n & \log \mu_n( (\cX \setminus A) \times \cY) \\
\notag & =
 \max \Big\{
 \limsupn  \tfrac1n \log \mu_n^1(\cX \setminus A),
\limsupn \tfrac1n \log \mu_n^2(\cX \setminus A)
\Big \} \\
& = \limsupn  \tfrac1n \log \mu_n^1(\cX \setminus A) \le - \inf I(\cX \setminus A),
\end{align}
$\bullet$ \eqref{eqn:lower_bound_A_times_B} follows by separating two cases (as either $y\in B$ or $y\notin \overline B$): 

If $y \notin \overline B$, 
then $W_n \cap B = \emptyset$ and so $\mu_n(A\times B) = \mu_n^1(A)\nu_n(B)$ for large $n$, whence 
\begin{align}
 \liminfn &  \tfrac1n \log \mu_n( A \times B)  = \liminfn 
	\bigg( 
		 \tfrac1n \log \mu_n^1(A) + \tfrac1n \log \nu_n(B)
	\bigg). 
	\label{eqn:lower_bound_A_times_B_2}
\end{align}

Suppose that $y \in  B$, i.e., $W_n\subset B$ for large $n$. 
By \eqref{eqn:liminf_sum_and_max} we obtain 
\begin{align}
\notag \liminfn &  \tfrac1n \log \mu_n( A \times B) \\
\notag  = \max \Big\{
& \liminfn \bigg(  \tfrac1n \log \mu_n^1(A) +\tfrac1n \log  \Big(\int_{W_n} \alpha_n \D \nu_n \Big) \bigg), \\
\notag & \liminfn  \bigg(\tfrac1n \log \mu_n^2(A) + \tfrac1n \log \Big( \int_{W_n} (1- \alpha_n) \D \nu_n  \Big) \bigg),  \\
 & \liminfn 
	\bigg( 
		 \tfrac1n \log \mu_n^1(A) + \tfrac1n \log \Big( \int_{W_n^c} \1_B \D \nu_n \Big)
	\bigg)
  \Big \} 
\end{align}
Using that $\liminfn \tfrac1n \log (\int_{W_n^c} \1_B \D \nu_n) \le 0$  together with 
\eqref{eqn:liminf_lambda_n_integrals_equal_zero} and
\eqref{eqn:logaritmic_inequality_mu_1_and_mu_2_liminf}, we obtain 
\begin{align}
\liminfn  \tfrac1n \log \mu_n( A \times B) 
\notag & = \max \Big\{
\liminfn \tfrac1n \log \mu_n^1(A), 
\liminfn \tfrac1n \log \mu_n^2(A)
\Big \} \\
&  \ge - \inf I(A) .
  \label{eqn:lower_bound_A_times_B_1}
\end{align}
Because $\inf L(B) \ge 0$, we conclude \eqref{eqn:lower_bound_A_times_B}. 

We leave it to the reader to check that 
$\eta_n$ is the 
weakly continuous product regular conditional probability under $\mu_n$ with respect to $\pi$.
\end{proof}






\bibliographystyle{abbrv}

\end{document}